\documentclass[reqno]{amsart}	

\usepackage{enumerate}
\usepackage{graphicx}
\usepackage{nicefrac}
\usepackage{algorithm}
\usepackage{algpseudocode}
\usepackage{booktabs}
\usepackage{nicefrac}
\usepackage{amsaddr}
\usepackage{geometry}
\usepackage[onehalfspacing]{setspace}
\allowdisplaybreaks

\newcommand{\D}{\mathcal{D}}
\newcommand{\N}{\mathcal{N}}
\newcommand{\Q}{\mathcal{Q}}
\newcommand{\R}{\mathcal{R}}
\newcommand{\T}{\mathcal{T}}

\newcommand{\NN}{\mathbb{N}}
\newcommand{\PP}{\mathbb{Q}}
\newcommand{\RR}{\mathbb{R}}
\newcommand{\WW}{\mathbb{W}}
\newcommand{\XX}{\mathbb{X}}
\newcommand{\YY}{\mathbb{Y}}

\newcommand{\E} {\mathsf{E}}

\newcommand{\Iloc}{\mathcal{I}^{\mathrm{loc}}}
\newcommand{\WWloc}{\WW^{\mathrm{loc}}}
\newcommand{\Ploc}{\Pi^{\mathrm{loc}}}
\newcommand{\uloc}{u^{\mathrm{loc}}_{\WW}}
\newcommand{\ut}{\widetilde u_{\WW}}
\newcommand{\uWW}{u_{\WW}}
\newcommand{\yxi}{y_{\myvec{\xi}}}

\newcommand{\myvec}   [1] {\boldsymbol{\mathsf{#1}}} % vectors
\newcommand{\mymatrix}[1] {\boldsymbol{\mathsf{#1}}} % matrices
\newcommand{\myind}   [1] {\boldsymbol{#1}}          % multi-indices
\newcommand{\myref}   [1] {\widehat{#1}}             % everything on reference element
\newcommand{\norm}    [2] [\XX]{\|#2\|_{#1}}
\newcommand{\res}     [2][\WW] {\rho_{#1}(#2)}
\newcommand{\wt}      [1] {\widetilde{#1}}

\newtheorem{proposition}{Proposition}
\newtheorem{lemma}      {Lemma}

\theoremstyle{definition}

\newtheorem{example}    {Example}
\newtheorem{remark}     {Remark}

\newcommand{\myState}[1]{\State\parbox[t]{\dimexpr\linewidth-\algorithmicindent}{#1\strut}}
\newcommand{\myStateDouble}[1]{\State\parbox[t]{\dimexpr\linewidth-\algorithmicindent-\algorithmicindent}{#1\strut}}

\begin{document}

\title[$hp$-adaptivitiy based on locally predicted error reduction]{An $hp$-adaptive strategy based on locally predicted error reductions}

\author[P.~Bammer, A.~Schr\"oder]{Patrick Bammer \and Andreas Schr\"oder}

\address{Fachbereich Mathematik, Paris Lodron Universit\"at Salzburg,\\ Hellbrunner~Str.~14, 5020 Salzburg, Austria}

\author[T.~P.~Wihler]{Thomas P.~Wihler}

\address{Mathematisches Institut, Universit\"{a}t Bern, Sidlerstr.~5, CH-3012 Bern, Switzerland}
		
\begin{abstract}
We introduce a new $hp$-adaptive strategy for self-adjoint elliptic boundary value problems that does \emph{not} rely on using classical a posteriori error estimators. Instead, our approach is based on a generally applicable prediction strategy for the reduction of the energy error that can be expressed in terms of local modifications of the degrees of freedom in the underlying discrete approximation space. The computations related to the proposed prediction strategy involve low-dimensional linear problems that are computationally inexpensive and highly parallelizable. The mathematical building blocks for this new concept are first developed on an abstract Hilbert space level, before they are employed within the specific context of $hp$-type finite element discretizations. For this particular framework, we discuss an explicit construction of $p$-enrichments and $hp$-refinements by means of an appropriate constraint coefficient technique that can be employed in any dimensions.  The applicability and effectiveness of the resulting $hp$-adaptive strategy is illustrated with some $1$- and $2$-dimensional numerical examples.
\end{abstract}


\thanks{%
A.~Schr\"{o}der acknowledges the support by the Bundesministerium f\"{u}r Bildung, Wissenschaft und Forschung (BMBWF) under the Sparkling Science project SPA 01-080 'MAJA -- Mathematische Algorithmen für Jedermann Analysiert'.\\
\indent T.~P.~Wihler acknowledges the financial support of the Swiss National Science Foundation (SNSF), Grant No. 200021\_212868. 
}


\maketitle
	
% -----------------------------------------------------------------------------------------
%    Section: Introduction
% -----------------------------------------------------------------------------------------
	
\section{Introduction}

Over the last few decades significant contributions have been made in 
%the field of the finite element method; especially, dealing with 
\emph{a posteriori error} estimation and automatic (adaptive) mesh refinement techniques
for finite element methods (FEM), see e.g.~\cite{ainsworth:1996, becker:2001, eriksson:1995, houston:2002, szabo:1991, verfurth:1996}. A majority of works in the adaptive FEM literature are focused exclusively on local (usually isotropic) refinements of elements (i.e. on so-called $h$-refinement). In recent years, however, motivated by the pioneering \emph{a priori} results on spectral and variable-order (so-called $hp$-version) FEM of Babu\v{s}ka and co-authors, cf.~\cite{babuska:1988, babuska:1987, babuska:1989, babuska:1994, GuoBabuska:1986b,GuoBabuska:1986a}, the adaptive selection of locally varying polynomial degrees (i.e. $p$-refinement) has been taken into account as well. In combination with local mesh adaptivity, such approaches lead to \emph{$hp$-adaptive} schemes; see e.g.~the overview article~\cite{MM:2014} and the references therein, or the monographs~\cite{schwab:1998,karniadakis:1999,demkowicz:2007,Solin:2004} for theoretical and practical aspects of $hp$-FEM. A key advantage of $hp$-approximations is their ability to approximate local singularities in (otherwise smooth) solutions of partial differential equations at high algebraic or even exponential convergence rates, see, e.g.~\cite{schwab:1998}.

While it seems reasonable to generalize classical $h$-version \emph{a posteriori} error estimation to the $hp$-framework, it is well-known that the derivation of effective computable $hp$-version error bounds involves considerable technical challenges~\cite{MW:2001,DEV:2016}. Furthermore, in addition to flagging elements for refinement (as in $h$-adaptive FEM), the design of $hp$-adaptive strategies requires the development of carefully devised decision procedures that allow to choose between various possible $hp$-refinements of the marked elements; this can be accomplished, for instance, by means of exploiting appropriate smoothness testing strategies~\cite{HoustonSuli:2005,EM:2007,wihler:2011,FW:2014}.

In contrast, the approach proposed in the present paper does neither rely on classical \emph{a posteriori} error estimators nor on smoothness indicators. Instead, we develop a \emph{prediction procedure} for an efficient reduction of the (global) energy error that can be represented explicitly in terms of  $hp$-refinements on local discrete spaces. Our methodology is closely related to the energy minimization technique~\cite{houston:2016}, where local problems are employed to compare various elementwise (competitive) $p$- and $hp$-refinements with regards to the potential contribution they may provide to the decay of the global error (see also~\cite{demkowicz:2007} for a related strategy). Inspired by the low-order approach presented in~\cite{HSW:2021}, the main novelty in the current paper consists in a new construction of the approximation spaces for the local problems that \emph{properly conforms to the weak formulation} of the underlying partial differential equation, and thus, provides a mathematically sound structure; this is opposed to~\cite{houston:2016}, where artificial conditions and numerical fluxes along patch boundaries are involved. 

More precisely, for the purpose of this work, in order to perform a (local) $hp$-refinement on an element~$Q$ in the mesh (with associated polynomial degree $p_Q$), we first decompose the numerical solution $u_{hp}$ on the current (global) $hp$-space into a local part $u_{hp}^{\text{loc}}$ (with local support on or around~$Q$), and a remaining (globally supported) part $\widetilde u_{hp} = u_{hp} - u_{hp}^{\text{loc}}$. Here, our goal is to suitably replace the local part $u_{hp}^{\text{loc}}$ on a modified local $hp$-space that allows for an improved resolution of possible small-scale features of the global solution on~$Q$. To this end, we consider locally supported spaces (spanned by so-called  \emph{enrichment functions}), which either provide a polynomial degree $>p_Q$ on $Q$, or an $hp$-refinement of $Q$ (involving a small number of subelements of~$Q$ and an appropriate polynomial degree distribution on this subelements), thereby giving rise to a \emph{$p$-enrichment} or an \emph{$hp$-replacement} on $Q$, respectively. The span of the enrichment functions together with the remaining part of the current solution, $\widetilde u_{hp}$, defines the local \emph{enrichment} or \emph{replacement space}, on which a potentially refined approximate solution is obtained. Since we explicitly include the remaining part of the current solution in the definition of the local enrichment/replacement space, the new approximate solution represents a \emph{global} solution that allows to rigorously predict the (global) error reduction resulting from replacing the current solution by the new one. Since the local enrichment/replacement spaces are low-dimensional (in comparison to the full $hp$-space), we emphasize that the locally predicted (but globally effective) error reduction can be computed at a negligible cost.  Our $hp$-adaptive algorithm passes through all elements of a current mesh (which can be done in parallel), and compares the predicted error reductions for different $p$-enrichments and $hp$-refinements in order to find an optimal enrichment on each element. Then, by using an appropriate marking strategy,  all those elements, from which the most substantial (global) error reduction can be expected, will be refined.

The paper is structured as follows: In \S\ref{sec:abstract_form} we begin by introducing an abstract  framework for linear elliptic problems in Hilbert spaces, and derive some general results on (global) error reductions, which are exploited to devise the error prediction strategy (in terms of low-dimensional modifications) in our new $hp$-adaptive scheme. In \S\ref{sec:appl_to_hp_FEM}, the abstract setting is applied to $hp$-FEM, whereby we specify two types of enrichment functions on each element of a given $hp$-mesh, namely $p$- and $hp$-enrichment functions; in particular, we present in detail the explicit construction of such enrichments on individual elements based on a recent constraint coefficients technique, see~\cite{schroeder:2008, byfut:2017}. Furthermore, the proposed $hp$-adaptive procedure is outlined in terms of a general algorithm in~\S\ref{sec:adaptive_alg}, and some numerical experiments for $hp$-type finite element discretizations in 1D and 2D, which illustrate the effectiveness of our approach, will be presented in~\S\ref{sec:numerical_examples}. 

Throughout this article, let $\NN = \lbrace 1,2,\ldots\rbrace$ denote the positive integers and set $\NN_0 := \NN \cup \lbrace 0\rbrace$. For any $n\in\NN$ we write $\underline{n}$ and $\underline{n}_0$ for the finite sets $\lbrace 1,2,\ldots, n\rbrace$ and $\lbrace 0,1,\ldots, n\rbrace$, respectively.

% -----------------------------------------------------------------------------------------
%    Section: Abstract setting
% -----------------------------------------------------------------------------------------
	
\section{Abstract framework}\label{sec:abstract_form}

On a (real) Hilbert space $\XX$ we consider a bounded, symmetric and coercive bilinear form $a:\XX\times\XX\to\RR$ that induces a norm $v\mapsto\norm{v} := a(v,v)^{\nicefrac12}$ on $\XX$. We also introduce a bounded linear form $b: \XX\to\RR$. 
Then, for any closed subspace $\WW\subseteq\XX$, by the Riesz representation theorem, there is a unique $u_{\WW}\in\WW$ such that the weak formulation
\begin{align}\label{eq:solW}
 a(u_{\WW},w)
 = b(w) \qquad\forall \, w\in\WW
\end{align}
holds true. Equivalently, defining the residual
\begin{align}\label{eq:res}
 \res{v} 
 := b(v) - a(u_{\WW},v), \qquad v\in\XX,
\end{align}
we have $\res{v} = 0$ for all $v\in\WW$. Furthermore, in line with the above notation, we signify by $u_{\XX}\in\XX$ the (full space) solution of the weak formulation
\begin{align}\label{eq:fullspace}
 a(u_{\XX},v)
 = b(v) \qquad \forall \, v\in\XX.
\end{align}

% -----------------------------------------------------------------------------------------
%    Subsection: Low-dimensional enrichments
% -----------------------------------------------------------------------------------------

\subsection{Low-dimensional enrichments}\label{subsec:low_dim_enrichments}

Consider the special case where the subspace $\WW\subseteq\XX$ introduced above is spanned by finitely many basis functions $\phi_1,\ldots,\phi_N\in\WW$, with $N\ge 1$, i.e.
\[
\WW=\operatorname{span}\{ \phi_1,\ldots,\phi_N \},\qquad \operatorname{dim}(\WW)=N<\infty.
\] 
Moreover, for an index subset $\Iloc\subset\underline{N}$, we let
\begin{equation}\label{eq:Wloc}
\WWloc:=\operatorname{span}\big\lbrace\phi_i:\,i\in\Iloc \big\rbrace\subset\WW;
\end{equation}
for instance, in the specific context of $hp$-discretizations to be discussed later on, the spaces $\WW$ and $\WWloc$ will take the roles of an $hp$-finite element space and a \emph{locally supported} subspace, respectively. We can then define a linear projection operator
\begin{align}\label{eq:linear_projOp}
\Ploc:\,\WW\to\WWloc,\qquad
v = \sum_{i\in\underline{N}} v_i \, \phi_i \mapsto \Ploc v := \sum_{i\in\Iloc} v_i \, \phi_i,
\end{align}
which enables a decomposition of the solution $u_{\WW}\in\WW$ of~\eqref{eq:solW} into a low dimensional and a remaining part 
\[
\uloc:=\Ploc u_{\WW}\in\WWloc,\qquad
\ut:=u_{\WW}-\Ploc u_{\WW},
\]
respectively, with 
\begin{equation}\label{eq:dw}
u_{\WW} = \uloc + \ut. 
\end{equation}
We aim to improve the approximation $u_{\WW}$ of the (full space) solution $u_{\XX}\in\XX$ of~\eqref{eq:fullspace} by enriching or replacing the local space $\WWloc$ by a subspace
\begin{align}\label{eq:Y}
 \YY
 := \operatorname{span}\{ \ut,\xi_1,\ldots,\xi_L \} \subseteq \XX
\end{align}
of  dimension $L+1$, where
\begin{align}\label{eq:xi}
\myvec{\xi}:=\{ \xi_1,\ldots,\xi_L \}\subseteq\XX
\end{align} 
is a (small) set of linearly independent elements in $\XX$. If $\WWloc\subset\YY$ then we call $\YY$ a \emph{local enrichment space}, otherwise $\YY$ is a \emph{local replacement space}; the functions $\xi_1,\ldots,\xi_L$ will in both cases be referred to as \emph{(local) enrichment functions}.  

We now introduce the low-dimensional problem
\begin{align}\label{eq:solY}
 u_{\YY}\in\YY : \qquad 
 a(u_{\YY},v) = b(v) \qquad \forall \, v\in \YY.
\end{align}
Since $\dim\YY = L+1$, there is a unique $\epsilon\in\RR$ and some uniquely determined $y_{\myvec{\xi}}\in \operatorname{span}\{ \xi_1,\ldots,\xi_L \}$ such that we can express the solution $u_{\YY}$ of~\eqref{eq:solY} in terms of a linear combination
\begin{align}\label{eq:uY}
 u_{\YY}
 =(1+\epsilon) \, \ut + y_{\myvec{\xi}}.
\end{align}
Then, defining the errors
\begin{align*}
 e_{\WW} := u_{\XX} - u_{\WW},
 \qquad
 e_{\YY} := u_{\XX} - u_{\YY},
\end{align*}
we prove the following result for the \emph{predicted error reduction} $\Delta e_{\WW,\YY}$, given by
\begin{align*}
 \Delta e_{\WW,\YY}^2 := \norm{e_{\WW}}^2 - \norm{e_{\YY}}^2.
\end{align*}

\begin{proposition}[Predicted error reduction]\label{pr:res}
For the predicted error reduction we have the identities
\begin{align}\label{eq:errid}
 \Delta e_{\WW,\YY}^2
=
\norm{u_{\YY}-u_{\WW}}^2
-2 \, \res[\YY]{\uloc}
 =\res{y_{\myvec{\xi}}}
 -\res[\YY]{\uloc},
\end{align}
with~$y_{\myvec{\xi}}\in\operatorname{span}\{ \xi_1,\ldots,\xi_L \}$ from~\eqref{eq:uY}, the residual~$\res{\cdot}$ defined in~\eqref{eq:res}, and $\res[\YY]{\cdot}:=b(\cdot) - a(u_{\YY},\cdot)$. 
\end{proposition}

\begin{proof}
We begin by noting the identity $e_{\YY}+(u_{\YY}-u_{\WW})=e_{\WW}$, 
which implies that $\norm{e_{\YY}+(u_{\YY}-u_{\WW})}^2
=\norm{e_{\WW}}^2$.
Thus, using that the bilinear form $a(\cdot,\cdot)$ induces the norm $\norm{\cdot}$, we deduce that
\begin{align*}
\norm{e_{\YY}}^2
+2 \, a(e_{\YY},u_{\YY}-u_{\WW})
+\norm{u_{\YY}-u_{\WW}}^2
=
\norm{e_{\WW}}^2.
\end{align*}
By Galerkin orthogonality, exploiting that $u_{\YY}-\ut\in\YY$ and $\uloc\in\WW$, we observe that
\[
a(e_{\YY},u_{\YY}-\ut)=0\qquad\text{and}\qquad
a(e_{\WW},\uloc)=0.
\]
Hence, recalling \eqref{eq:dw} we obtain
\[
a(e_{\YY},u_{\YY}-u_{\WW})
=a(e_{\YY},u_{\YY}-\ut)-a(e_{\YY},\uloc)
=-a(e_{\YY},\uloc).
\]
Furthermore, invoking~\eqref{eq:fullspace}, it follows that
\[
a(e_{\YY},u_{\YY}-u_{\WW})
=-a(u_{\XX},\uloc)+a(u_{\YY},\uloc)
=-b(\uloc)+a(u_{\YY},\uloc)
=-\res[\YY]{\uloc}.
\]
Thus,
\[
\norm{e_{\YY}}^2
-2 \, \res[\YY]{\uloc}
+\norm{u_{\YY}-u_{\WW}}^2
=
\norm{e_{\WW}}^2,
\]
%We begin by noting the identity $e_{\YY}+(u_{\YY}-\ut)=e_{\WW}+\uloc$, cf.~\eqref{eq:dw},
%which implies that
%\[
%\norm{e_{\YY}+(u_{\YY}-\ut)}^2
%=\norm{e_{\WW}+\uloc}^2.
%\]
%Thus, using that the bilinear form $a(\cdot,\cdot)$ induces the norm $\norm{\cdot}$, we deduce that
%\begin{align*}
%\norm{e_{\YY}}^2
%+2a(e_{\YY},u_{\YY}-\ut)
%+\norm{u_{\YY}-\ut}^2
%=
%\norm{e_{\WW}}^2
%+2a(e_{\WW},\uloc)
%+\norm{\uloc}^2.
%\end{align*}
%By Galerkin orthogonality, exploiting that $u_{\YY}-\ut\in\YY$ and $\uloc\in\WW$, we observe that
%\begin{equation}\label{eq:torth}
%a(e_{\YY},u_{\YY}-\ut)=0\qquad\text{and}\qquad a(e_{\WW},\uloc)=0.
%\end{equation}
%Therefore,
%\[
%\norm{e_{\YY}}^2
%+\norm{u_{\YY}-\ut}^2
%=
%\norm{e_{\WW}}^2
%+\norm{\uloc}^2,
%\]
which yields the first identity in~\eqref{eq:errid}. 
%Moreover, recalling~\eqref{eq:dw} and~\eqref{eq:uY}, we infer that
%\begin{align*}
%\norm{u_{\YY}-\ut}^2
%-\norm{\uloc}^2
%&=a((u_{\YY}-\ut)-\uloc,(u_{\YY}-\ut)+\uloc)\\
%&=a(u_{\YY}-u_{\WW},u_{\YY}+\uloc-\ut)\\
%&=a(u_{\YY}-u_{\WW},y_{\myvec{\xi}}+\uloc+\epsilon\ut).
%\end{align*}
%Since $\ut\in\WW\cap\YY$, we notice the orthogonality property
%\[
%a(u_{\WW}-u_{\YY},\ut)=0,
%\]
%which leads to
%\begin{equation}\label{eq:c1}
%\norm{u_{\YY}-\ut}^2
%-\norm{\uloc}^2
%=a(u_{\YY}-u_{\WW},y_{\myvec{\xi}})+a(u_{\YY}-u_{\WW},\uloc).
%\end{equation}
%Then, employing~\eqref{eq:solY} with the test function $y_{\myvec{\xi}}\in\YY$, we have
%\begin{equation}
%a(u_{\YY}-u_{\WW},y_{\myvec{\xi}})=b(y_{\myvec{\xi}})-a(u_{\WW},y_{\myvec{\xi}})=\res[\WW]{y_{\myvec{\xi}}}.
%\end{equation}
%Similarly, applying~\eqref{eq:solW} with the test function~$\uloc\in\WW$, we arrive at
%\begin{equation}\label{eq:c3}
%a(u_{\YY}-u_{\WW},\uloc)
%=a(u_{\YY},\uloc)-b(\uloc)
%=-\res[\YY]{\uloc},
%\end{equation}
%Combining~\eqref{eq:c1}--\eqref{eq:c3} gives the second identity in~\eqref{eq:errid}.
Moreover, recalling~\eqref{eq:dw} and~\eqref{eq:uY}, we have $u_{\YY}-u_{\WW}
=\epsilon \, \ut+y_{\myvec{\xi}}-\uloc$, and since $\ut\in\WW\cap\YY$, we notice the orthogonality property
\[
a(u_{\YY}-u_{\WW},\ut)=0.
\]
Thus,
\begin{align*}
\norm{u_{\YY}-u_{\WW}}^2
&=a(u_{\YY}-u_{\WW},u_{\YY}-u_{\WW})\\
&=a(u_{\YY}-u_{\WW},\epsilon \, \ut+y_{\myvec{\xi}}-\uloc)\\
&=a(u_{\YY}-u_{\WW},y_{\myvec{\xi}})
-a(u_{\YY}-u_{\WW},\uloc).
\end{align*}
Then, employing~\eqref{eq:solY} with the test function $y_{\myvec{\xi}}\in\YY$, we have
\[
a(u_{\YY}-u_{\WW},y_{\myvec{\xi}})=b(y_{\myvec{\xi}})-a(u_{\WW},y_{\myvec{\xi}})=\res[\WW]{y_{\myvec{\xi}}}.
\]
Similarly, applying~\eqref{eq:solW} with the test function~$\uloc\in\WW$, we arrive at
\[
-a(u_{\YY}-u_{\WW},\uloc)
=-a(u_{\YY},\uloc)+b(\uloc)
=\res[\YY]{\uloc}.
\]
Combining the above equalities gives the second identity in~\eqref{eq:errid}.
\end{proof}

\begin{remark}[Energy reduction property]
If the space $\YY$ is a local enrichment of $\WWloc$, i.e. if $\WWloc\subset\YY$, then $\res[\YY]{\uloc}=0$, and the identities~\eqref{eq:errid} simplify to
\begin{align*}
 \Delta e_{\WW,\YY}^2
 = \norm{u_{\YY}-u_{\WW}}^2
 =\res{y_{\myvec{\xi}}};
\end{align*}
in particular, in this case, the error resulting from the local enrichment will \emph{not increase}, or will even \emph{decrease} (i.e.~$\norm{e_{\WW}}^2 > \norm{e_{\YY}}^2$) if $u_{\YY}\neq u_{\WW}$. In analogous terms, introducing the energy functional 
\begin{align*}
 \E(v) := \frac{1}{2} \, a(v,v) - b(v),\qquad v\in\XX,
\end{align*}
and recalling the Dirichlet variational principle, we note that
\begin{align*}
 \E(u_{\WW}) = \min_{w\in\WW} \E(w)\qquad\text{and}\qquad
 \E(u_{\YY}) = \min_{v\in\YY} \E(v),
\end{align*}
whence we immediately deduce the \emph{energy reduction} property $\E(u_{\YY}) \leq \E(u_{\WW})$ for the solutions $u_{\WW}\in\WW$ and $u_{\YY}\in\YY$ from~\eqref{eq:solW} and~\eqref{eq:solY}, respectively. In the general case, when $\YY$ is a local replacement space, a minor modification of the proof of Proposition~\ref{pr:res} yields the identity
\begin{align*}
 \Delta e_{\WW,\YY}^2
 = \norm{u_{\YY}-\ut}^2-\norm{\uloc}^2,
\end{align*}
which shows that the error is guaranteed to decrease, i.e. $\Delta e_{\WW,\YY}^2 > 0$, whenever $\norm{u_{\YY}-\ut}^2 > \norm{\uloc}^2$.
\end{remark}

% -----------------------------------------------------------------------------------------
%    Subsection: Linear algebra representation
% -----------------------------------------------------------------------------------------

\subsection{Linear algebra representation}

We will now illustrate how the difference of the residuals $\res{y_{\myvec{\xi}}}$ and $\res[\YY]{\uloc}$ from Proposition~\ref{pr:res} can be computed by means of linear algebra. To this end, for $\{\xi_1,\ldots,\xi_L\}$ from~\eqref{eq:xi}, we introduce the matrix $\mymatrix{A}=(A_{ij})\in\RR^{L\times L}$ with entries
\begin{subequations}\label{eq:localq}
\begin{align}
 A_{ij}
 &:= a(\xi_j,\xi_i), \qquad i,j\in\underline{L},
\intertext{and the (column) vectors $\myvec{b},\myvec{c}\in\RR^L$, which are given component-wise by}
 b_i& := b(\xi_i), \quad 
 c_i := a(\ut,\xi_i), 
\qquad i\in\underline{L},
\intertext{respectively. Furthermore, we let}
a_{00}&:=\norm{\uWW}^2-\norm{\uloc}^2 - 2 \, \delta,
\qquad\text{with}\quad
\delta := b(\uloc)-\norm{\uloc}^2.
\end{align}
\end{subequations}

\begin{remark}
In the finite element context, where the functions $\xi_1,\ldots,\xi_L$ and $\uloc$ are supposedly local, we notice that all of the quantities defined in~\eqref{eq:localq} are inexpensive to compute (except for the global term $\norm{\uWW}$ which, however, needs to be evaluated once only).
\end{remark}

\begin{proposition}[Low-dimensional computation of error reductions]\label{pr:repres_error}
Suppose that $\dim\YY = L+1$ (in particular, we implicitly assume that $\ut\neq 0$), and consider the (symmetric) linear system
\begin{equation}\label{eq:linsys}
\begin{pmatrix}
a_{00} & \myvec{c}^\top\\
\myvec{c} & \myvec{A}
\end{pmatrix}
\begin{pmatrix}
\epsilon\\ \myvec{y}
\end{pmatrix}
=
\begin{pmatrix}
\delta\\
\myvec{b}-\myvec{c}
\end{pmatrix},
\end{equation}
with the solution components $\epsilon\in\mathbb{R}$ and~$\myvec{y}\in\mathbb{R}^L$. Then, the predicted error change from Proposition~\ref{pr:res} can be computed by means of the formula
\begin{align}\label{eq:formula}
 \Delta e_{\WW,\YY}^2
 =\myvec{y}^\top(\myvec{b}-\myvec{c})-\norm{\uloc}^2+\epsilon \, \delta.
 \end{align}
\end{proposition}

\begin{proof}
Recalling the decomposition~\eqref{eq:dw}, we begin by noticing that
\[
\norm{\ut}^2
=\norm{\uWW-\uloc}^2
=\norm{\uWW}^2+\norm{\uloc}^2 - 2 \, a(\uWW,\uloc),
\]
which, upon applying~\eqref{eq:solW} with the test function $v=\uloc$, results in
\begin{equation}\label{eq:a00}
\norm{\ut}^2
=
\norm{\uWW}^2+\norm{\uloc}^2 - 2 \, b(\uloc)=a_{00}.
\end{equation}
Similarly, with $v=\ut$ in~\eqref{eq:solW}, we infer that
\begin{align*}
b(\ut)-a(\ut,\ut)
&=a(\uWW,\ut)-a(\ut,\ut)
=a(\uloc,\ut)
=a(\ut,\uloc)
=a(\uWW,\uloc)-a(\uloc,\uloc)\\
&=b(\uloc)-\norm{\uloc}^2,
\end{align*}
from which we arrive at
\begin{equation}\label{eq:d}
\delta = b(\ut)-a(\ut,\ut)=b(\ut)-\norm{\ut}^2.
\end{equation}
Then, using~\eqref{eq:solY} with~$v=\ut\in\YY$, and applying the representation~\eqref{eq:uY}, leads to
\[
\delta =
b(\ut)-a(\ut,\ut)
=a(u_{\YY},\ut)-a(\ut,\ut)
=\epsilon \, a(\ut,\ut)+a(\yxi,\ut)
=\epsilon \, \norm{\ut}^2+a(\yxi,\ut),
\]
and equivalently, due to~\eqref{eq:a00},
\begin{equation}\label{eq:linsys1}
\epsilon \,  a_{00}+a(y_{\myvec{\xi}},\ut) = \delta.
\end{equation}
In addition, for $i\in\underline{L}$, testing~\eqref{eq:solY} with $v = \xi_i$, and exploiting~\eqref{eq:uY}, 
yields
\begin{equation}\label{eq:linsys2}
(1+\epsilon) \, a(\ut,\xi_i)+a(y_{\myvec{\xi}},\xi_i)=b(\xi_i)\qquad\forall \, i\in\underline{L}.
\end{equation}
Therefore, applying the linear combination 
\begin{equation}\label{eq:liny}
y_{\myvec{\xi}} = \sum_{j\in\underline{L}} y_j \, \xi_j,
\end{equation} 
for a uniquely defined coefficient vector $\myvec{y} := (y_1,\ldots,y_L)^{\top}\in\RR^L$, the equations~\eqref{eq:linsys1} and~\eqref{eq:linsys2} transform into
\begin{align*}
\epsilon \, a_{00}+\sum_{j\in\underline{L}} a(\ut,\xi_j) \, y_j &= \delta,
\intertext{and}
\epsilon \, a(\ut,\xi_i) + \sum_{j\in\underline{L}} a(\xi_j,\xi_i) \, y_j
 &= b(\xi_i) -a(\ut,\xi_i) \qquad \forall \, i\in\underline{L},
\end{align*}
which is the linear system~\eqref{eq:linsys}.
Moreover, owing to~\eqref{eq:dw}, it holds that
\[
\res{y_{\myvec{\xi}}}
=b(y_{\myvec{\xi}})-a(u_{\WW},y_{\myvec{\xi}})
=b(y_{\myvec{\xi}})-a(\ut,y_{\myvec{\xi}})-a(\uloc,y_{\myvec{\xi}}).
\]
In addition, invoking~\eqref{eq:uY} and~\eqref{eq:solW}, we have
\begin{align*}
-\res[\YY]{\uloc}
&=a(u_{\YY},\uloc)-b(\uloc)\\
&=(1+\epsilon) \,  a(\ut,\uloc)+a(\yxi,\uloc)-a(\uWW,\uloc)\\
&=-a(\uloc,\uloc) + \epsilon \, a(\ut,\uloc)+a(\yxi,\uloc)\\
&=-\norm{\uloc}^2+\epsilon \Big( a(\uWW,\ut)-a(\ut,\ut) \Big) +a(\yxi,\uloc)\\
&=-\norm{\uloc}^2+\epsilon \Big( b(\ut)-\norm{\ut}^2 \Big) + a(\yxi,\uloc).
\end{align*}
Recalling~\eqref{eq:d} implies that 
\begin{align*}
 -\res[\YY]{\uloc}
=-\norm{\uloc}^2+\epsilon \, \delta + a(\yxi,\uloc).
\end{align*}
Thus, upon involving the linear combination~\eqref{eq:liny}, we finally obtain
\begin{align*}
 \res{y_{\myvec{\xi}}} - \res[\YY]{\uloc}
 &= b(y_{\myvec{\xi}}) - a(\ut,y_{\myvec{\xi}})
- \norm{\uloc}^2 + \epsilon \, \delta 
 = \sum_{j\in\underline{L}} y_j \, \big( b(\xi_j) - a(\ut,\xi_j) \big) - \norm{\uloc}^2 + \epsilon \, \delta,
 %&=\myvec{y}^\top(\myvec{b}-\myvec{c})-\norm{\uloc}^2 + \epsilon \, \delta.
\end{align*}
which, in view of~\eqref{eq:errid}, leads to the asserted identity~\eqref{eq:formula}.
\end{proof}

% -----------------------------------------------------------------------------------------
%    Subsection: Assembling aspects
% -----------------------------------------------------------------------------------------

\subsection{Assembling aspects}\label{subsec:assembling_aspects}

We aim to apply the abstract framework developed in the previous sections to the finite element context. For this purpose, we let $\XX=\XX(\Omega)$ be a Hilbert function space defined over a bounded open domain $\Omega\subset\mathbb{R}^d$, $d\ge 1$. We present the assembling of the matrix $\mymatrix{A}$ and the vectors $\myvec{b}$, $\myvec{c}$  occurring in the linear system~\eqref{eq:linsys} in the specific case where $\WW\subseteq\XX$ is a subspace of finite dimension $N := \dim\WW$ with the basis $\{ \phi_1,\ldots,\phi_N \}$. Let $\D$ be a decomposition of $\Omega$ into closed subsets $K\subseteq\Omega$, i.e.
\begin{align*}
 \overline{\Omega} = \bigcup_{K\in\D} K,
 % \qquad
 % \text{and $K\cap K'=\emptyset$ for any $K,K'\in\D$ with $K\neq K'$.}
\end{align*} 
where $\operatorname{int}(K)\cap \operatorname{int}(K')=\emptyset$ for any $K,K'\in\D$ with $K\neq K'$. For any $K\in\D$, consider the (local) Hilbert space consisting of all restrictions of functions $v\in\XX$ to $K$, i.e.~$\XX_K := \{ v_{ | \, K} : v\in\XX \}$. Moreover, let $\big\{ \zeta_1^K,\ldots,\zeta_{M_K}^K \big\}\subseteq \XX_K$ be a set of functions such that there exist representation matrices $\mymatrix{C}_K = \big(\, c_{ij}^K \, \big)\in\RR^{N\times M_K}$ and $\mymatrix{D}_K = \big( \, d_{ij}^K \, \big)\in\RR^{L\times M_K}$ satisfying
\begin{align}
  \phi_{i \, | \, K} &= \sum_{j\in\underline{M_K}} c_{ij}^K \, \zeta_j^K, \qquad i\in\underline{N}, \label{eq:rep_phi_xi_01} \\
  \xi_{i \, | \, K}  &=  \sum_{\in\underline{M_K}} d_{ij}^K \, \zeta_j^K, \qquad i\in\underline{L}, \label{eq:rep_phi_xi_02}
\end{align}
with $\{\xi_1,\ldots,\xi_L\}$ from~\eqref{eq:xi}. Finally, let the bilinear form $a(\cdot,\cdot)$ as well as the linear form $b(\cdot)$ be decomposable in the sense that
\begin{align}\label{eq:dec_a_b}
 a(v,w) = \sum_{K\in\D} a_K\big( v_{ | \, K}, w_{ | \, K} \big), \qquad
 b(v) = \sum_{K\in\D} b_K\big( v_{ | \, K} \big)
 \qquad \forall \, v,w\in\XX,
\end{align}
for some (local) bilinear forms $a_K : \XX_K\times\XX_K\to\RR$ and (local) linear forms $b_K : \XX_K\to\RR$. For any $K\in\D$ we introduce the (local) matrix $\mymatrix{A}_K = \big(\, A_{ij}^K \, \big)\in\RR^{M_K\times M_K}$ with entries
\begin{align*}
 A_{ij}^K := a_K \big(\zeta_j^K, \zeta_i^K \big), \qquad i,j\in\underline{M_K},
\end{align*}
and the (local) vector $\myvec{b}_K = \big(b_i^K\big)\in\RR^{M_K}$, which is given component-wise by
\begin{align*}
 b_i^K := b_K\big(\zeta_i^K \big), \qquad i\in\underline{M_K}.
\end{align*}

Let $\wt{\myvec{u}} = (u_1,\ldots,u_N)^{\top}\in\RR^N$ denote the uniquely determined coefficient vector of $\ut$ from~\eqref{eq:dw} with respect to the basis $\{ \phi_1,\ldots,\phi_N \}$ of $\WW$, i.e. $\ut = \sum_{i\in\underline{N}} u_i \, \phi_i$. Note that $u_i=0$ for $i\in\Iloc$. Then the matrix $\mymatrix{A} \in\RR^{L\times L}$ and the vectors $\myvec{b},\myvec{c}\in\RR^L$ from the linear system \eqref{eq:linsys} can be assembled by the (local) quantities $\mymatrix{A}_K,\mymatrix{C}_K, \mymatrix{D}_K$, and $\myvec{b}_K$ as the following result shows.

\begin{proposition}[Assembling]\label{prop:assembling}
The identities
\begin{align}\label{eq:assembling_1}
 \mymatrix{A} = \sum_{K\in\D} \mymatrix{D}_K \mymatrix{A}_K \mymatrix{D}_K^{\top}, \qquad
 \myvec{b} = \sum_{K\in\D} \mymatrix{D}_K \myvec{b}_K,\qquad
 \myvec{c} = \sum_{K\in\D} \mymatrix{D}_K \mymatrix{A}_K \mymatrix{C}_K^{\top} \wt{\mymatrix{u}}
\end{align}
hold true.
\end{proposition}

\begin{proof}
Let $[\cdot]_{ij}$ and $[\cdot]_i$ denote the components of a matrix or a vector, respectively. Then, it holds that
\begin{align*}
 \left[ \sum_{K\in\D} \mymatrix{D}_K \mymatrix{A}_K \mymatrix{D}_K^{\top} \right]_{ij}
 = \sum_{K\in\D} \big[\mymatrix{D}_K\mymatrix{A}_K\mymatrix{D}_K^{\top}\big]_{ij}
 &= \sum_{K\in\D} \Bigg( \sum_{\ell\in\underline{M_K}} \sum_{k\in\underline{M_K}} d_{ik}^K \, a_K\big(\zeta_{\ell}^K, \zeta_{k}^K\big) \, d_{j\ell}^K \Bigg) \\
 &= \sum_{K\in\D} a_K \hspace{-0,075cm} \Bigg( \sum_{\ell\in\underline{M_K}} d_{j\ell}^K \, \zeta_{\ell}^K, \sum_{k\in\underline{M_K}} d_{ik}^K \, \zeta_{k}^K \Bigg) \\
 &= \sum_{K\in\D} a_K\big( \xi_{j \, | \, K}, \xi_{i \, | \, K} \big) \\
 &= a(\xi_j, \xi_i)
 = A_{ij},
\end{align*}
for all $i,j\in\underline{L}$, cf.~\eqref{eq:rep_phi_xi_02} and \eqref{eq:dec_a_b} which is the first identity. The second identity follows from
\begin{align*}
 \left[ \sum_{K\in\D} \mymatrix{D}_K\myvec{b}_K\right]_i
 = \sum_{K\in\D} \left[ \mymatrix{D}_K\myvec{b}_K \right]_i
 &= \sum_{K\in\D} \Bigg( \sum_{j\in\underline{M_K}} d_{ij}^K \, b_K\big(\zeta_j^K\big) \Bigg) \\
 &= \sum_{K\in\D} b_K \hspace{-0,075cm} \Bigg( \sum_{j\in\underline{M_K}} d_{ij}^K \, \zeta_j^K \Bigg) \\
 &= \sum_{K\in\D} b_K\big( \xi_{i \, | \, K} \big)
 = b(\xi_i) = b_i,
\end{align*}
for $i\in\underline{L}$. The last identity follows in an analogous way.
\end{proof}

% -----------------------------------------------------------------------------------------
%    Section: Application to hp-finite element spaces
% -----------------------------------------------------------------------------------------

\section{Application to $hp$-finite element spaces}\label{sec:appl_to_hp_FEM}

For $d\ge 1$, let $\Omega\subset\RR^d$ be a bounded interval (if $d=1$), or a bounded and open set with a Lipschitz boundary $\Gamma := \partial\Omega$ that is composed of a finite number of straight faces (if $d\ge 2$). Furthermore, we consider a boundary part  $\Gamma_D\subseteq\Gamma$ of positive surface measure, and introduce an associated Hilbert space
\begin{align}\label{eq:H1}
 \XX := \big\lbrace v\in H^1(\Omega) : v = 0 \text{ on } \Gamma_D \big\rbrace,
\end{align}
where $H^1(\Omega)$ denotes the usual Sobolev space of all functions in $L^2(\Omega)$, with weak first-order partial derivatives in $L^2(\Omega)$. 

The goal of the following subsections is to define appropriate $hp$-finite element approximation subspaces in $\XX$. For this purpose, we will begin by introducing a family of hierarchical polynomial spaces on the  $d$-dimensional hypercube signified by $\myref{Q} := [-1,1]^d$, which will be referred to as the \emph{reference element}.

% -----------------------------------------------------------------------------------------
%    Subsection: Hierachical polynomials on the reference element
% -----------------------------------------------------------------------------------------

\subsection{Polynomial spaces on the reference element} \label{subsec:hirachical_poly_on_Q}

For any $j\in\NN_0$ we introduce the $1$-dimensional functions $\psi_j:[-1,1]\to\RR$, given by
\begin{align}\label{eq:defi_1D_psi}
 \psi_0(t) := \frac{1}{2} \, (1-t), \qquad
 \psi_1(t) := \frac{1}{2} \, (1+t), \qquad
 \psi_j(t) := \int_{-1}^t L_{j-1}(t)\,\mathsf{d}t, \quad j\geq 2,
\end{align}
where $L_j:[-1,1]\to\RR$ denotes the $j$-th Legendre polynomial, normalized such that $L_j(-1)=(-1)^j$, $j\ge 1$; we recall the \emph{Bonnet recursion formula}
\begin{align*}
 L_0(t) := 1,\qquad 
 L_1(t) := t,\qquad
 j \, L_{j}(t) = (2j-1) \, t \, L_{j-1}(t) - (j-1) \, L_{j-2}(t), \quad j\ge 2, 
\end{align*}
as well as the relation
\begin{align*}
 \psi_j(t)=\frac{L_j(t)-L_{j-2}(t)}{2j-1},
% L_j(t) := \frac{2j-3}{j} \, t \, L_{j-1}(t) - \frac{j-3}{j} \, L_{j-2}(t), 
\quad j\geq 2;
\end{align*}
see, e.g., \cite[\S3.1 \& A.4]{szabo:1991}. The basis functions $\psi_0$ and $\psi_1$ are referred to as \emph{nodal or external shape functions} whereas the high-order polynomials $\psi_j$, for $j\ge 2$, are called \emph{internal modes} (based on the fact that $\psi_j(\pm1)=0$ for $j\ge 2$). We display the functions $\psi_0,\ldots,\psi_4$ in the following Figure~\ref{fig:integrLP}.

\begin{figure}
	\centering
	\includegraphics[scale=0.4]{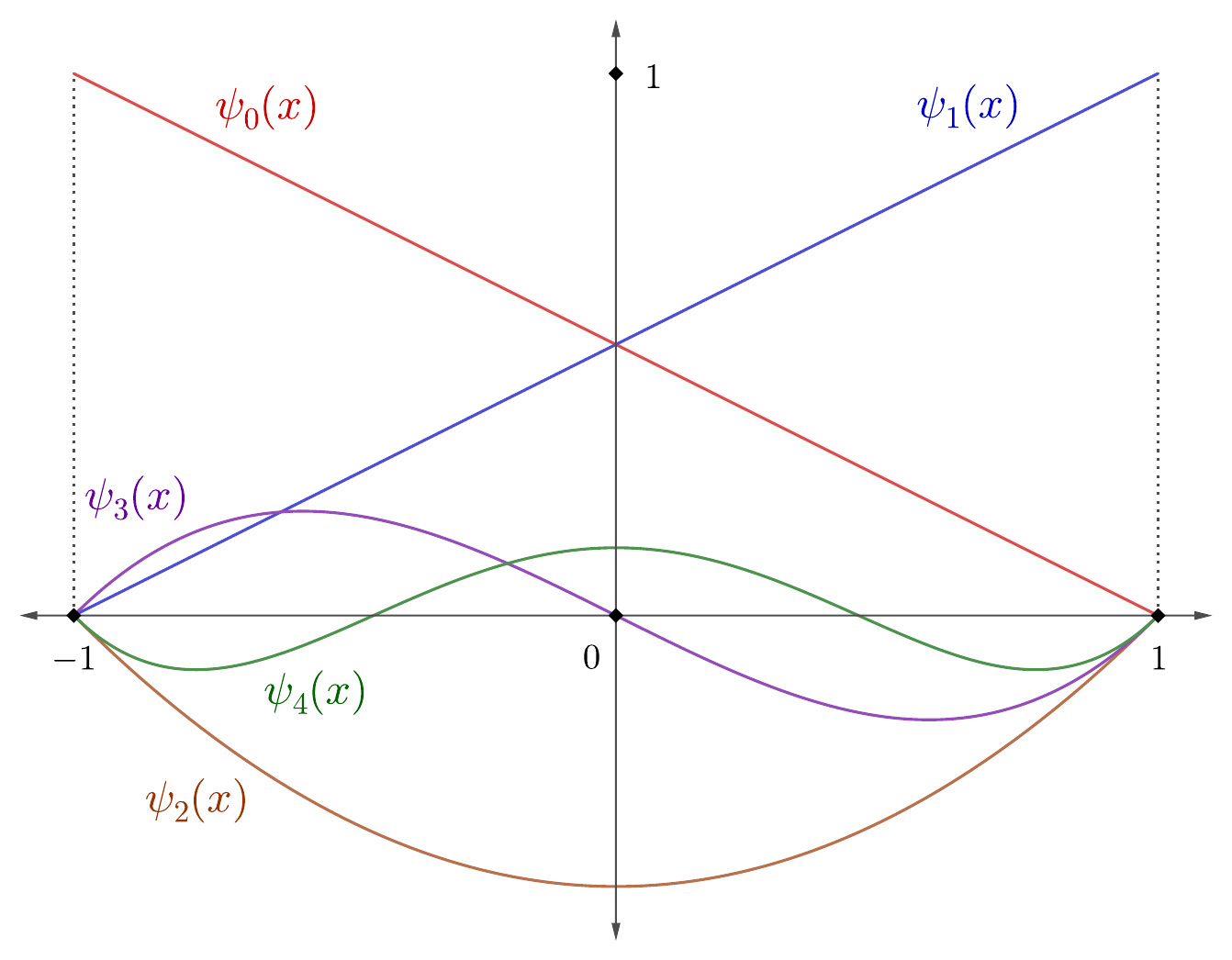} 
	\caption{\it The functions $\psi_0,\ldots,\psi_4$ on $[-1,1]$.}
	\label{fig:integrLP}
\end{figure} 

Furthermore, in the multi-dimensional case, for any multi-index $\myind{j}=(j_1,\ldots,j_d)\in\NN_0^d$, we define the functions $\myref{\psi}_{\myind{j}}:\myref{Q}\to\RR$ by
\begin{align}\label{eq:defi_dDim_psi}
 \myref{\psi}_{\myind{j}}(\myvec{x}) := \prod_{k\in\underline{d}} \psi_{j_k}(x_k), \qquad
 \myvec{x} = (x_1,\ldots,x_d)^{\top}\in\myref{Q}.
\end{align}
Moreover,  for any $k\in\NN$, let $\PP_k(\myref{Q}) := \operatorname{span}\big\lbrace \myref{\psi}_{\myind{j}} : \myind{j}\in \underline{k}_0^d \big\rbrace$ signify the usual space of all polynomials up to degree $k\in\NN$ in each coordinate direction on the reference element $\myref{Q}$.

% -----------------------------------------------------------------------------------------
%    Subsection: hp-finite elememnt space
% -----------------------------------------------------------------------------------------

\subsection{$hp$-finite element space} 

The vertices of the reference element $\myref{Q}$ will be labeled in terms of a multi-index notation. Specifically, to each $\myind{i}\in\lbrace 0,1\rbrace^d$, we associate a corresponding vertex
\begin{align*}
 \myref{\myvec{v}}_{\myind{i}} := \big( 2i_1-1,\ldots, 2i_d-1 \big)^{\top}.
 \qquad 
\end{align*}
Similarly, we consider physical elements $Q\subset\Omega$ whose corner points are indexed by $\myvec{v}_{\myind{i}}\in\RR^d$, $i=1,\ldots, 2^d$; we call $Q$ a \emph{transformed hexahedron} if the map $F_Q : \myref{Q}\to Q$ defined by
\begin{align}\label{eq:elem_transform}
 F_Q(\myvec{x}) := \sum_{|\myind{i}|\leq 1} \myref{\psi}_{\myind{i}}(\myvec{x}) \, \myvec{v}_{\myind{i}},
 \qquad \myvec{x} = (x_1,\ldots,x_d)^{\top}\in\myref{Q},
\end{align}
is bijective, where we write $|\myind{i}| := i_1 + \cdots + i_d$ to signify the order of a multi-index $\myind{i}=(i_1,\ldots,i_d)\in\NN_0^d$. Note that it holds $F_Q(\myref{\myvec{v}}_{\myind{i}}) = \myvec{v}_{\myind{i}}$ for all $\myind{i}\in\lbrace 0,1\rbrace^d$.  

\begin{example}[$F_Q$ for $d=2$]
%We display a specific example for $d=2$ in Figure~\ref{fig:elemTransf}. 
In this case, the transformation $F_Q$ from~\eqref{eq:elem_transform} takes the form
\begin{align*}
 F_Q(\myvec{x}) 
  &= \myref{\psi}_{00}(\myvec{x}) \, \myvec{v}_{00} + \myref{\psi}_{10}(\myvec{x}) \, \myvec{v}_{10} + \myref{\psi}_{11}(\myvec{x}) \, \myvec{v}_{11} + \myref{\psi}_{01}(\myvec{x}) \, \myvec{v}_{01},\qquad
  \myvec{x}=(x_1,x_2)^{\top}\in\myref{Q},
\end{align*}
and the tensor product functions $\myref{\psi}_{\myind{i}}$ are given by
\begin{align*}
 \myref{\psi}_{00}(\myvec{x}) &= \psi_0(x_1) \, \psi_0(x_2), \qquad
 \myref{\psi}_{10}(\myvec{x}) = \psi_1(x_1) \, \psi_0(x_2), \\
 \myref{\psi}_{11}(\myvec{x}) &= \psi_1(x_1) \, \psi_1(x_2), \qquad
 \myref{\psi}_{01}(\myvec{x}) = \psi_0(x_1) \, \psi_1(x_2).
\end{align*}
\end{example}
\begin{figure}
	\centering
	\includegraphics[scale=0.5]{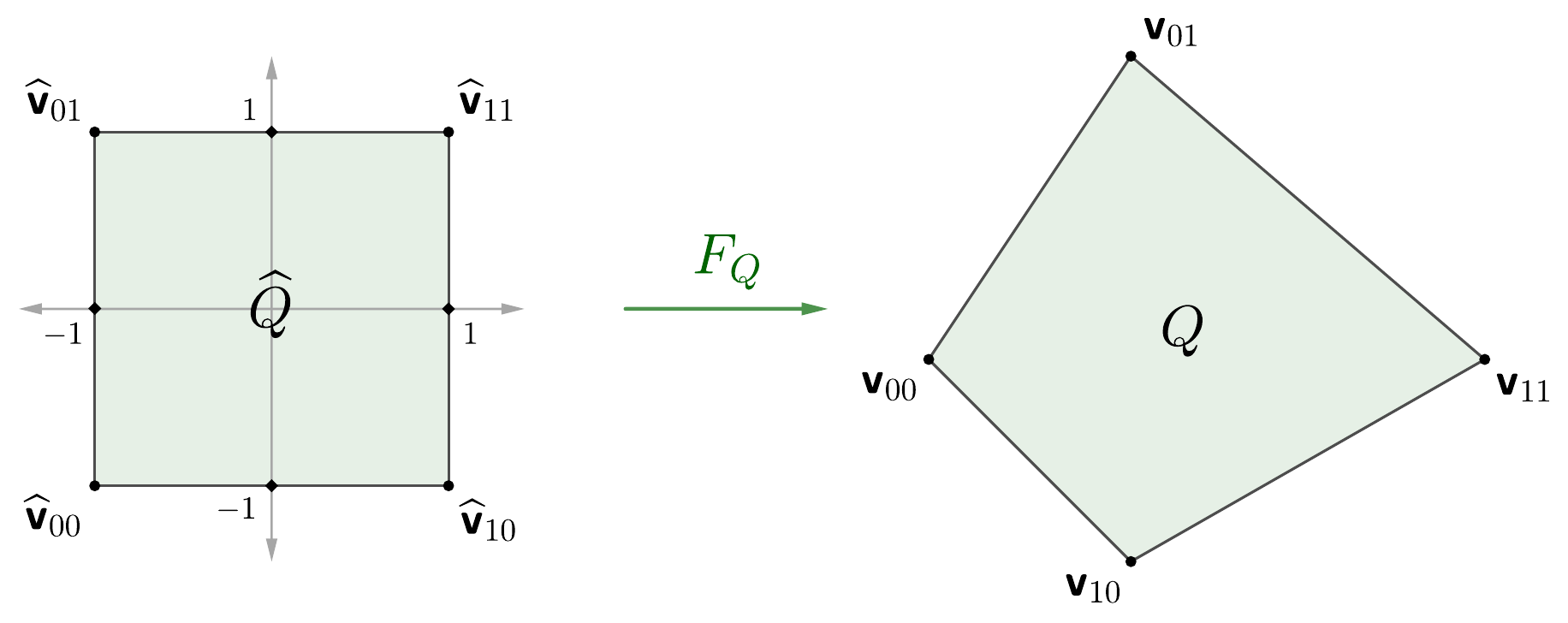} 
	\caption{\it The bijective mapping $F_Q:\myref{Q}\to Q$ for the $2$-dimensional case.}
	\label{fig:elemTransf}
\end{figure}

For the purpose of introducing finite element subspaces $\WW \subseteq \XX$, cf.~\eqref{eq:H1}, following our abstract framework in~\S\ref{subsec:assembling_aspects}, we let $\Q$ be a decomposition of $\Omega$  into transformed hexahedrons with the following additional property: 
If $Q_1\cap Q_2\neq \emptyset$ for $Q_1,Q_2\in\Q$ with $Q_1\neq Q_2$, then $Q_1\cap Q_2$ represents a $(d-r)$-dimensional face of $Q_1$ or $Q_2$ for some $r\in\underline{d}$; 
we call any faces of dimension zero, one and two \emph{vertices}, \emph{edges} and \emph{faces}, respectively. Furthermore, let us assume that there is no change of the type of boundary conditions within the $(d-r)$-dimensional faces of one element $Q\in\Q$. We now define the $hp$-finite element space
\begin{align}\label{eq:FE_space}
 \WW := \big\lbrace v\in\XX : v_{ | \, Q}\circ F_Q\in\PP_{p_Q}(\myref{Q}) \quad \forall \, Q\in\Q \big\rbrace,
\end{align}
associated with the decomposition $\Q$ of $\Omega$, where $p_Q\ge 1$ represents an (isotropic) local polynomial degree on each element $Q\in\Q$.

% -----------------------------------------------------------------------------------------
%    Subsubsection: Element mapping

\subsection{Constraint Coefficients}\label{subsec:constraint_coeff}

In the adaptive finite element procedure we shall represent the functions $\myref{\psi}_{\myind{i}}$ from~\eqref{eq:defi_dDim_psi}, defined on the reference element $\myref{Q}$, in terms of functions defined on sub-hexahedra of $\myref{Q}$. These representations can be computed efficiently by means of so-called \emph{constraint coefficients}, see~\cite{byfut:2017, schroeder:2008}. To explain this concept, consider a sub-hexahedron
\begin{align}\label{eq:tensorT}
 T := \prod_{k\in\underline{d}} I_k\subset \myref{Q}, 
\end{align}
where $I_k=[a_k,b_k]$, $-1\le a_k<b_k\le 1$, are 1-dimensional intervals for each $k\in\underline{d}$. Then,  we define the functions
\begin{equation}\label{eq:tensorP}
\wt{\psi}_{\myind{j}}:\,T\to\mathbb{R},\qquad \wt{\psi}_{\myind{j}} := \myref{\psi}_{\myind{j}}\circ F_T^{-1},\qquad \myind{j}\in\NN_0,
\end{equation}
with the bijective element map $F_{T}:\myref{Q}\to T$, cf.~\eqref{eq:elem_transform}. Here, due to the tensor structure~\eqref{eq:tensorT} of $T$, we emphasize that $F_T$ is the composition of a dilation and a translation, wherefore we observe the identity
\begin{equation}\label{eq:psiT}
\wt{\psi}_{\myind{j}}=\prod_{k\in\underline{d}}\left(\psi_{j_k} \circ F_{I_k}^{-1}\right).
\end{equation}
In particular, we infer that the functions $\wt{\psi}_{\myind{j}} $ in~\eqref{eq:tensorP} constitute a polynomial basis on $T$. Therefore, there exist uniquely determined \emph{constraint coefficients} $b_{\myind{i}\myind{j}}^{T}\in\RR$ such that
\begin{align}\label{eq:representation_restriction}
 \myref{\psi}_{\myind{i} \, | \, T} 
 = \sum_{\myind{j}\leq \myind{i}} b_{\myind{i}\myind{j}}^T \, \wt{\psi}_{\myind{j}}
 = \sum_{\myind{j}\leq \myind{i}} b_{\myind{i}\myind{j}}^T \, \left(\myref{\psi}_{\myind{j}} \circ F_{T}^{-1}\right),
 \qquad \myind{i}\in\NN_0^d,
\end{align}
where the sums are taken over all multi-indices $\myind{j}=(j_1,\ldots,j_d)\in\NN_0^d$ with $j_k\leq i_k$ for all $k\in\underline{d}$. 

\begin{example}[$d=1$]
In Figure~\ref{fig:restr_integrLP} we illustrate how the restriction of the $1$-dimensional function $\psi_2$ to the interval $I = [-1,0]$ can be expressed in terms of the functions $\wt{\psi}_0, \wt{\psi}_1,$ and $\wt{\psi}_2$.
\begin{figure}
	\centering
	\includegraphics[scale=0.4]{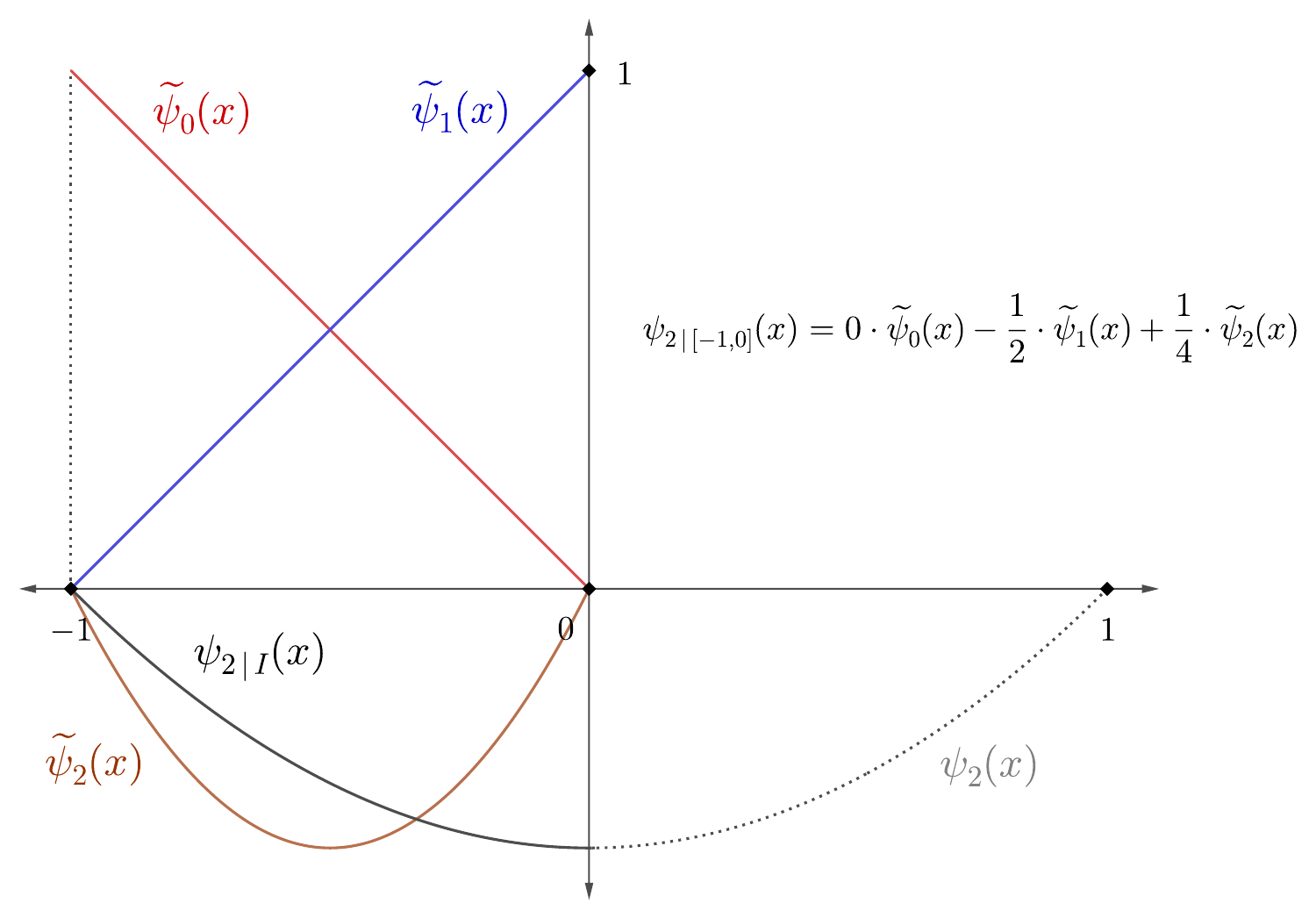} 
	\caption{\it The restriction of $\psi_2$ to $I$ in terms of the functions $\wt{\psi}_0, \wt{\psi}_1, \wt{\psi}_2$ in the $1$-dimensional case.}
	\label{fig:restr_integrLP}
\end{figure}
\end{example}

The ensuing result shows that the multi-dimensional constraint coefficients can be expressed in terms of tensor-products of the associated 1-dimensional quantities for which recursion formulas are stated in the Appendix.

\begin{lemma}
For $\myind{i},\myind{j}\in\NN_0^d$, the constraint coefficients $b_{\myind{i}\myind{j}}^T$  from \eqref{eq:representation_restriction} are given by
\begin{align*}
 b_{\myind{i}\myind{j}}^T = \prod_{k\in\underline{d}} b_{i_k, j_k}^{I_k},
\end{align*}
where $T$ is the sub-hexahedra from~\eqref{eq:tensorT} and $ b_{i_k, j_k}^{I_k}$ are the uniquely determined $1$-dimensional constraint coefficients from
\begin{align*}
 \psi_{i_k \, | \, I_k} = \sum_{j_k = 0}^{i_k} b_{i_k, j_k}^{I_k} \, \left(\psi_{j_k} \circ F_{I_k}^{-1}\right).
\end{align*}
\end{lemma}
\begin{proof}
The argument is based on exploiting the tensor structure of $\myref{Q}$ and $T$, and of the functions $\myref{\psi}_{\myind{i}}$. Indeed, for $\myind{i}=(i_1,\ldots,i_d)\in\NN_0^d$, using~\eqref{eq:defi_dDim_psi} and~\eqref{eq:tensorT}, and applying the representation~\eqref{eq:representation_restriction} in the 1-dimensional case, we obtain that
\begin{align*}
 \psi_{\myind{i} \, | \,T} 
  = \prod_{k\in\underline{d}} \psi_{i_k \, | \, I_k}
  = \prod_{k\in\underline{d}} \Bigg( \sum_{j_k\leq i_k} b_{i_k, j_k}^{I_k} \, \left(\psi_{j_k} \circ F_{I_k}^{-1}\right) \Bigg).
\end{align*}
Then, rearranging terms yields
\begin{align*}
 \psi_{\myind{i} \, | \,T}  
  = \sum_{\myind{j}\leq \myind{i}} \Bigg( \prod_{k\in\underline{d}} b_{i_k, j_k}^{I_k} \, \left(\psi_{j_k} \circ F_{I_k}^{-1} \right)\Bigg) 
  = \sum_{\myind{j}\leq \myind{i}} \Bigg( \prod_{k\in\underline{d}} b_{i_k, j_k}^{I_k} \prod_{k\in\underline{d}}\left(\psi_{j_k} \circ F_{I_k}^{-1} \right)\Bigg).
\end{align*}
Owing to~\eqref{eq:psiT}, we conclude that
\begin{align*}
 \psi_{\myind{i} \, | \,T} 
 = \sum_{\myind{j}\leq \myind{i}} \Bigg( \prod_{k\in\underline{d}} b_{i_k, j_k}^{I_k} \Bigg) \, \wt{\psi}_{\myind{j}}.
\end{align*}
The assertion follows from the uniqueness of the coefficients in \eqref{eq:representation_restriction}.
\end{proof}

% -----------------------------------------------------------------------------------------
%    Subsection: Enrichment functions on a single refinment mesh
% -----------------------------------------------------------------------------------------

\subsection{Enrichment functions}\label{subsec:enrichmentF}

We will characterize refinements of elements $Q\in\Q$ in a given mesh~$\Q$ via corresponding refinements of the reference element $\myref{Q}$. To this end, given a (fixed) point $\myref{\myvec{z}}\in(-1,1)^d$ in the interior of $\myref{Q}$, we call a decomposition $\R(\myref{Q})$ of $\myref{Q}$ into the $2^d$ sub-hexahedra
\begin{align*}
 \myref{T}_{\myind{i}} := \prod_{k\in\underline{d}} [ a_k^{\myind{i}}, b_k^{\myind{i}} ], \qquad \myind{i}=(i_1,\ldots,i_d)\in\lbrace 0,1\rbrace^d,
\end{align*}
with
\begin{align*}
 a_k^{\myind{i}} := \min \lbrace 2\, i_k -1, z_k \rbrace, \qquad
 b_k^{\myind{i}} :=  \max \lbrace 2 \, i_k - 1, z_k \rbrace, 
 \qquad k\in\underline{d},
\end{align*}
a \emph{refinement of $\myref{Q}$ with respect to $\myref{\myvec{z}} = (z_1,\ldots,z_d)^{\top}$}. 

\begin{figure}
	\centering
	\includegraphics[scale=0.55]{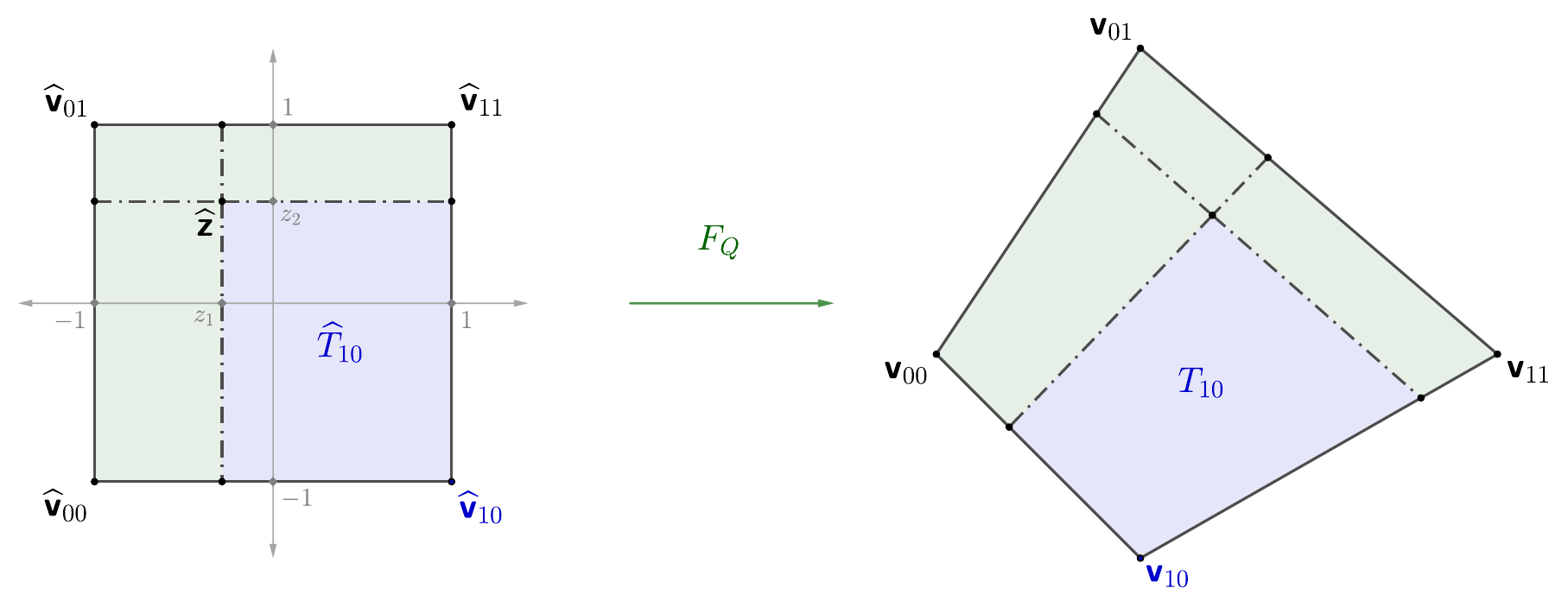} 
	\caption{\it A refinement of the reference element $\myref{Q}$ with respect to $\myref{\myvec{z}}=(z_1,z_2)^{\top}\in(-1,1)^2$ and the corresponding refinement of the element $Q = F_Q(\myref{Q})$.}
	\label{fig:refinment_referenceElem}
\end{figure}

Let us now focus on some element $Q\in\Q$. To introduce a refinement of $Q$, we consider first a refinement $\R(\myref{Q})$ of $\myref{Q}$ with respect to $\myref{\myvec{z}}\in(-1,1)^d$, as outlined above, and define
\begin{align*}
 \R(Q) := \big\lbrace T_{\myind{i}} : \myind{i}\in\lbrace 0,1\rbrace^d \big\rbrace, \qquad
 T_{\myind{i}} := F_{Q}(\myref{T}_{\myind{i}}).
\end{align*}
We display an illustration of the 2-dimensional situation in Figure~\ref{fig:refinment_referenceElem}. Note that the multi-index $\myind{i}$ encodes the location of the sub-hexahedra $T_{\myind{i}}$ with respect to the Cartesian reference coordinate system in $\myref{Q}$; in fact, the $k$-th entry of $\myind{i}$ classifies whether $\myref{T}_{\myind{i}}$ lies in the left (if $i_k = 0$) or the right (if $i_k=1$) half of $Q$ with respect to $\myref{\myvec{z}}$ along the $k$-th axial direction. Finally, let us denote by $\T$ the resulting mesh when replacing the element $Q$ by its refinement $\R(Q)$, i.e. $\T := \big( \Q\setminus\lbrace Q\rbrace \big) \cup \R(Q).$

% -----------------------------------------------------------------------------------------
%    Subsubsection: Enrichment functions

\subsubsection{Enrichment strategies on a single element}\label{sec:1element}

For an element $Q\in\Q$ with an associated polynomial degree $p_Q$, two types of local enrichment functions $\xi_1,\ldots,\xi_L$, cf.~\eqref{eq:xi}, will be considered:
\begin{enumerate}[(i)]
\item
For the definition of \emph{$p$-enrichment functions on $Q$}, we consider polynomials on the reference element $\widehat{Q}$, with polynomial degrees larger than $p_{Q}$, and transform them to the physical element $Q$. \medskip

\item For the construction of \emph{$hp$-enrichment functions on $Q$}, we consider polynomials on the sub-hexahedra $\myref{T}_{\myind{i}}\in\R(\myref{Q})$, $\myind{i}\in\lbrace 0,1\rbrace^d$, of a refinement $\R(\myref{Q})$ of $\myref{Q}$ with respect to some $\myref{\myvec{z}}\in(-1,1)^d$, which are then transformed to $Q$.
\end{enumerate}
We give an illustration of these two scenarios for the $p$-enrichment or $hp$-refinement of a 1-dimensional element in Figure~\ref{fig:p_vs_h_enrichment}. The polynomial functions resulting from the above mappings from $\myref{Q}$ or from sub-hexahedra of $\myref{Q}$ to $Q$ will be termed \emph{transformed polynomials}.
\begin{figure}
	\centering
	\includegraphics[scale=0.2]{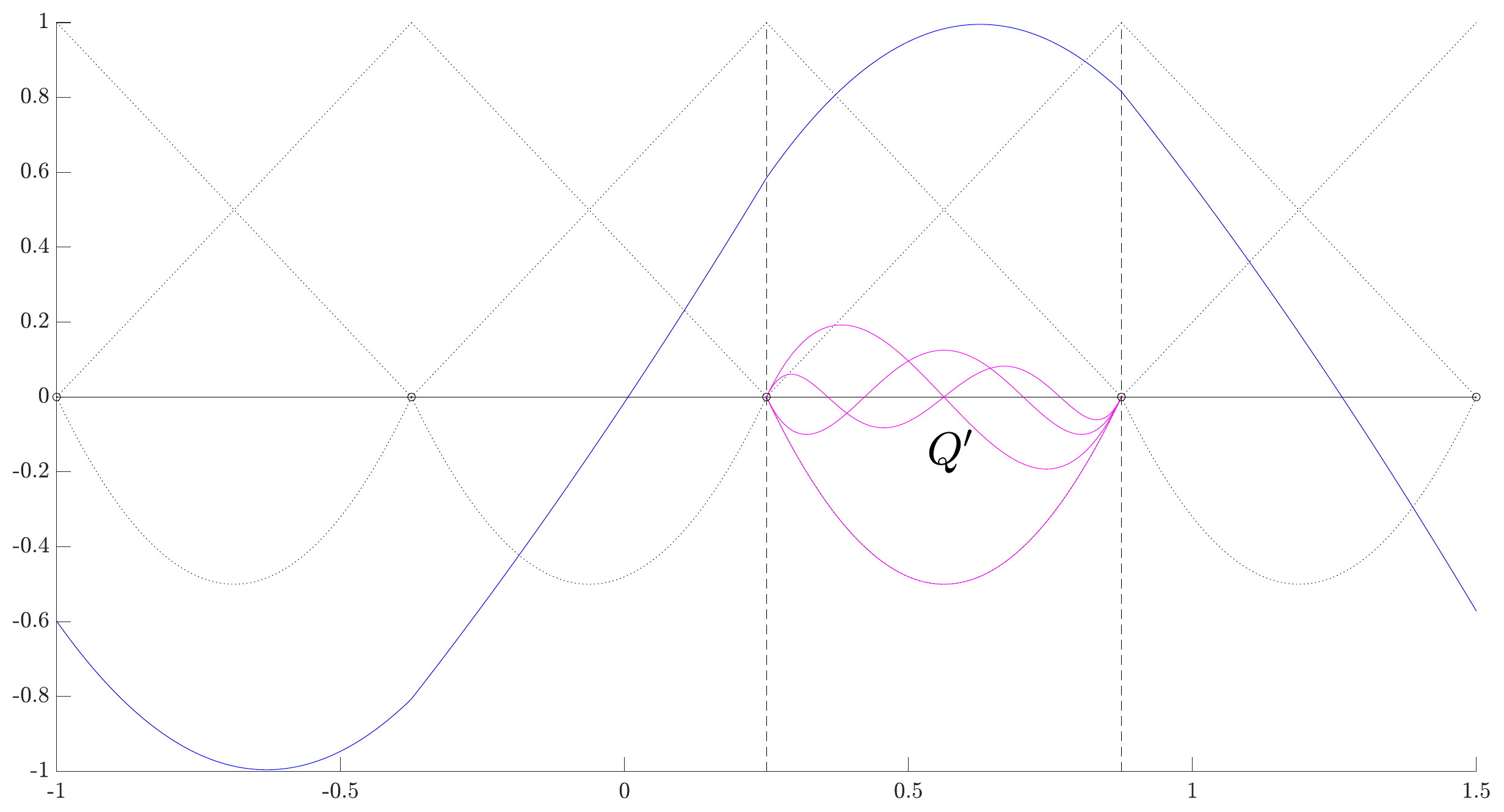} \hfill
	\includegraphics[scale=0.2]{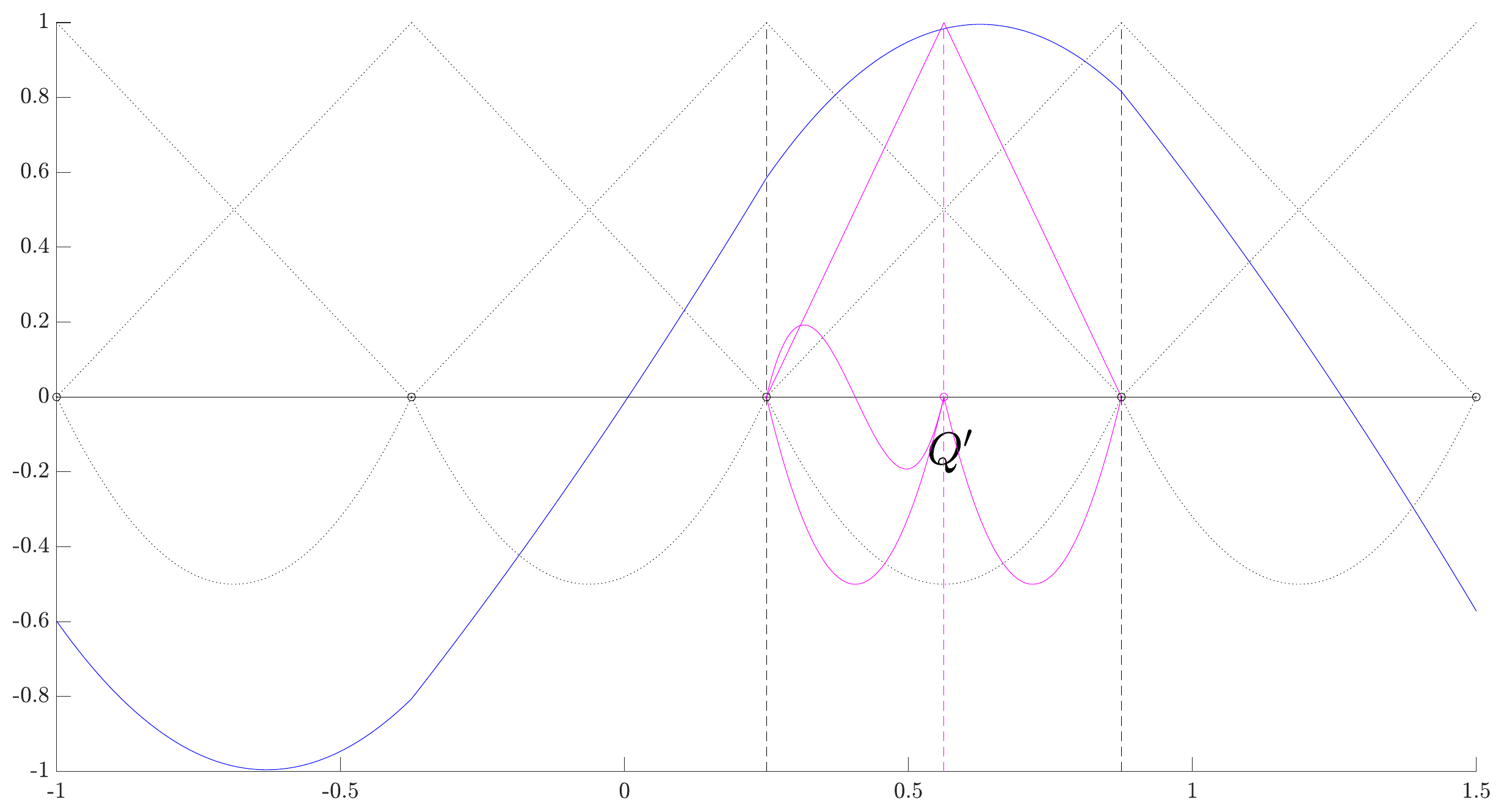}
	\caption{\it A $p$-enrichment (left) vs.~an $hp$-refinement (right). In both figures the dicrete solution $u_{\WW}$ is higlighted in blue, the in each case four enrichment functions are depicted in magenta and the basis $\lbrace \phi_1,\ldots,\phi_N\rbrace$ of $\WW$ is indicated in dotted lines.}
	\label{fig:p_vs_h_enrichment}
\end{figure}

In the $hp$-adaptive procedure described in \S\ref{sec:adaptive_alg} we aim to compare different $p$-enrichments and $hp$-refinements on $Q$; in particular \emph{competitive} refinements, cf.~\cite{houston:2016}, where different enrichments, which generate the same number of degrees of freedom, are compared with eachother in view of a maximal potential predicted error reduction. We will express both, the $hp$-enrichment functions (given on the subelements $T_{\myind{i}}\in\R(Q)$) as well as the $p$-enrichment functions (given on $Q$) in terms of the transformed polynomials $\zeta_{\myind{j}}^{\myind{i}} : T_{\myind{i}}\to\RR$, 
\begin{align}\label{eq:zetas}
 \zeta_{\myind{j}}^{\myind{i}}(\myvec{x}) := \myref{\psi}_{\myind{j}} \circ F_{\myind{i}}^{-1}(\myvec{x}), \qquad \myvec{x}\in T_{\myind{i}},
\end{align}
for any $\myind{j}\in\NN_0^d$ and $\myind{i}\in\lbrace 0,1\rbrace^d$. For ease of notation, we write $F_{\myind{i}}$ for the bijective mapping $ F_{T_{\myind{i}}}: \myref{Q}\to T_{\myind{i}}$, cf.~\eqref{eq:elem_transform}; similarly, we denote by $\mymatrix{A}_{\myind{i}}, \mymatrix{C}_{\myind{i}}, \mymatrix{D}_{\myind{i}}$, and $\myvec{b}_{\myind{i}}$ the local quantities $\mymatrix{A}_{T_{\myind{i}}}, \mymatrix{C}_{T_{\myind{i}}}, \mymatrix{D}_{T_{\myind{i}}}$ and $\myvec{b}_{T_{\myind{i}}}$ from~\S\ref{subsec:assembling_aspects}.
We emphasize that the quantities $\mymatrix{A}_{\myind{i}}, \mymatrix{C}_{\myind{i}}$ and $\myvec{b}_{\myind{i}}$, for \emph{each} of the different enrichments on $Q$ to be compared, need to be computed \emph{once only}.

% -----------------------------------------------------------------------------------------
%    Subsubsection: p-enrichment functions

\subsubsection{$p$-enrichments on $Q$}

For any multi-index $\myind{j}=(j_1,\ldots,j_d)\in\NN_0^d$, let us introduce the functions $\xi_{\myind{j}}: \Omega\to\RR$ by
\begin{align}\label{eq:defi_xi_p_enrichment}
 \xi_{\myind{j}}(\myvec{x}) := 
 \begin{cases} 
  \myref{\psi}_{\myind{j}} \circ F_{Q}^{-1}(\myvec{x}), & \text{if } \myvec{x}\in Q, \\
  0, & \text{if } \myvec{x}\in\Omega\setminus Q.
 \end{cases}
\end{align}
For the $p$-enrichments on $Q$ we choose a finite subset of the functions 
\begin{equation}\label{eq:penrich}
\mathfrak{E}_p = \lbrace \xi_{\myind{j}} : \myind{j}\in\myind{J}_d \rbrace,
\end{equation} 
for some index set
\begin{align*}%\label{eq:indexSet_pEnr}
 \myind{J}_d \subseteq \big\lbrace \myind{j}\in\NN_0^d : j_k \geq 2 \text{ for } k\in\underline{d} \big\rbrace, \qquad
 \text{with} \quad |\myind{J}_d| = L < \infty.
\end{align*}
We emphasize that we exclusively consider transformed (higher-order bubble-type) polynomials that vanish along the boundary of $Q$; cf. Remark~\ref{rem:patches} below for some generalizations to patches. In particular, we have the following result:

\begin{proposition}\label{prop:properties_penrichmentF}
Any $\xi\in\mathfrak{E}_p$ is continuous on $\Omega$, and it holds $\operatorname{supp}(\xi) = Q$.
\end{proposition}

Obviously, there are various possibilities to choose appropriate $p$-enrichment functions. For instance, if we select all transformed polynomials up to a certain polynomial degree $p_{\max}> p_Q$, we have $\myind{J}_d = \lbrace 2,\ldots, p_{\max}\rbrace^d$,
whereas for the so-called \emph{hierarchical surplus} we choose
\begin{align*}
 \myind{J}_d = \big\lbrace \myind{j}\in\lbrace 2,\ldots, p_{\max}\rbrace^d : |\myind{j}| \geq p_{Q}+2 \big\rbrace.
\end{align*}

% -----------------------------------------------------------------------------------------
%    Subsubsection: p-enrichment functions

\subsubsection{$hp$-enrichment functions on $Q$}\label{subsec:hp_enrichments}

In the case of $hp$-refinements, we construct local enrichment functions that can be associated with those $r$-dimensional faces of the refinement $\R(Q)$ which do not lie on the boundary $\partial Q$. We will call such faces \emph{internal nodes} of the refinement $\R(Q)$. In this sense, the \emph{vertex} $F_{Q}(\myvec{\myref{z}})$ is the only $0$-dimensional internal node of $\R(Q)$, \emph{edges} in the interior of $Q$ are the $1$-dimensional internal nodes of $\R(Q)$, and the \emph{elements} $T_{\myind{i}}$, for $\myind{i}\in\lbrace 0,1\rbrace^d$, represent the $d$-dimensional internal nodes of $\R(Q)$.

% -----------------------------------------------------------------------------------------

\paragraph*{Indexing of internal nodes}

Any $r$-dimensional internal node of $\R(Q)$, $r\in\underline{d}_0$, can be identified uniquely by  $r$ axial directions of $\myref{Q}$, which are represented by an \emph{orientation tuple} 
\begin{align*}
 \myind{a}\in
 D_r := \big\lbrace (a_1,\ldots,a_r)\in\underline{d}^r : a_1 < \ldots < a_r \big\rbrace,
\end{align*}
together with a \emph{location tuple} $\myind{\ell}\in\lbrace 0,1\rbrace^r$ that fixes its position with respect to the center point $F_Q(\myvec{\myref{z}})$. For any $\myind{a}=(a_1,\ldots,a_r)\in D_r$, let us denote by $ A(\myind{a})$ the set of its components, i.e. $A(\myind{a}) := \lbrace a_k : k\in\underline{r} \rbrace$.
Note that a tuple $\myind{a}\in D_r$ describes the orientation (and, implicitly, contains the dimension $r$) of an internal node with respect to the Cartesian reference coordinate system in $\myref{Q}=F_Q^{-1}(Q)$. Moreover, the $k$-th entry of the location tuple $\myind{\ell} = (\ell_1,\ldots,\ell_r)\in\lbrace 0,1\rbrace^d$ defines whether an internal node with orientation $\myind{a}=(a_1,\ldots,a_r)$ lies in the left (if $\ell_k=0$) or in the right (if $\ell_k=1$) half of the refinement $\R(Q)$ along the $k$-th axial direction of the reference coordinate system in $\myref{Q}=F_Q^{-1}(Q)$. Observe further that, for any $d\geq 1$, the only $0$-dimensional internal node of $\R(Q)$ is represented by the empty tuples $\myind{a} = ()$ and $\myind{\ell} = ()$.

\begin{example}[Internal nodes for the $3$-dimensional case]\label{ex:3}
The 0-dimension center point is given by the empty tuples $\myind{a}=()$ and $\myind{\ell}=()$. For the $1$-dimensional edges of $\R(\myref{Q})$ there are three orientations, namely parallel to any of the $x$-, $y$- or $z$-axes; they are described by $\myind{a}=(1)$, $\myind{a}=(2)$, and $\myind{a}=(3)$, respectively. Moreover, there are three orientations for $2$-dimensional faces, namely parallel to any of the $xy$-, $xz$- or $yz$-planes; they correspond to the orientation pairs $\myind{a}=(1,2)$, $\myind{a}=(1,3)$, and $\myind{a}=(2,3)$, respectively. Finally, there is a single orientation triple for the eight full-dimensional elements $\myref{T}_{\myind{i}}$ given by $\myind{a}=(1,2,3)$. For the interpretation of the location tuple, let us consider, for instance, the four internal nodes of dimension~2, which are parallel to the $xy$-plane, i.e. with orientation $\myind{a} = (1,2)$; they are highlighted in Figure~\ref{fig:interpretation_locations} (left), and their possible locations are represented by the location pairs listed in Table~\ref{tb:loc}.
To give a further example, the left $1$-dimensional internal node in $x$-direction of $\R(Q)$, highlighted in Figure~\ref{fig:interpretation_locations} (right) is characterized by $\myind{a} = (1)$ and $\myind{\ell}=(0)$.
\begin{figure}
	\centering\hfill
	\includegraphics[scale=0.175]{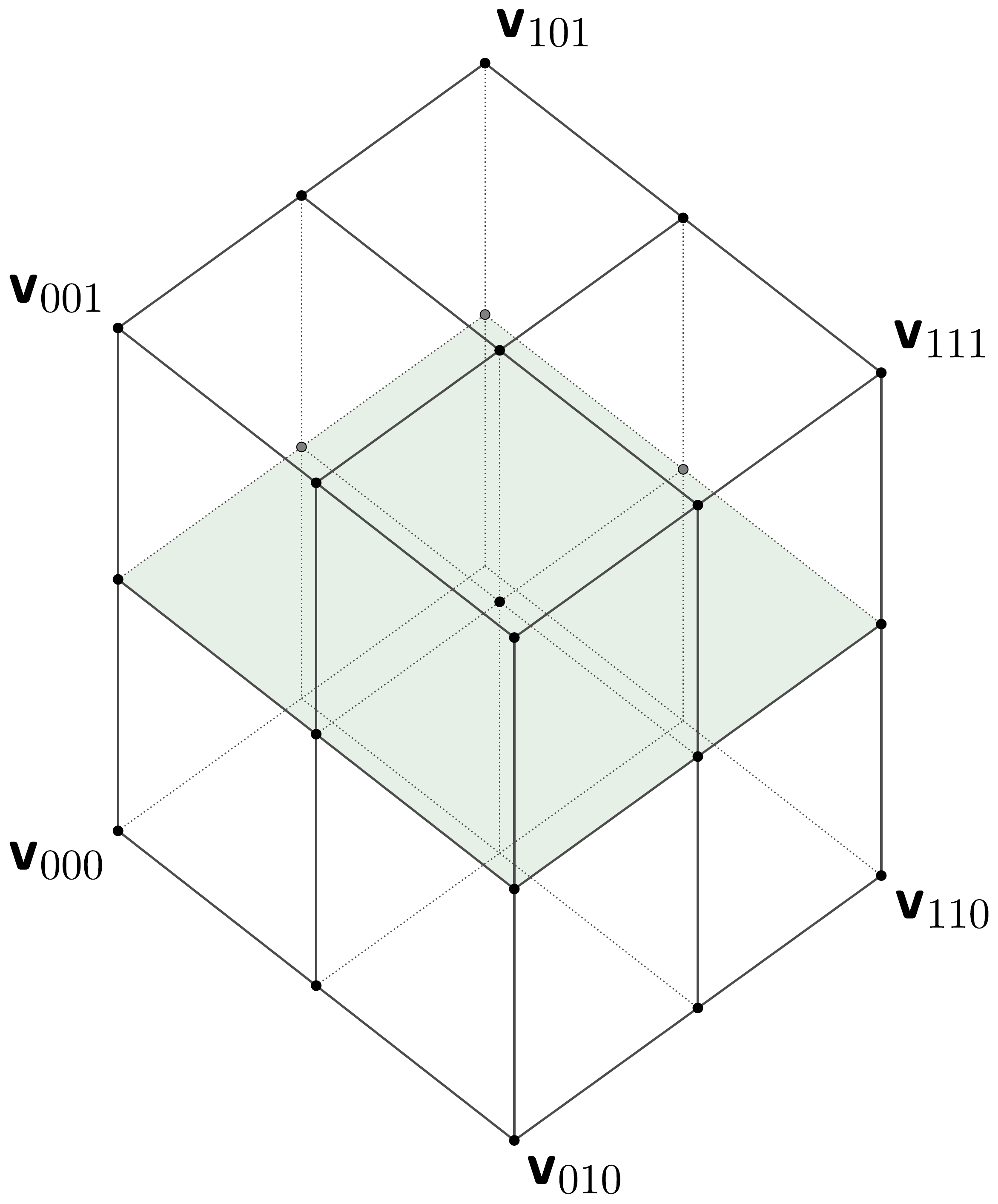}
	\hfill
	\includegraphics[scale=0.175]{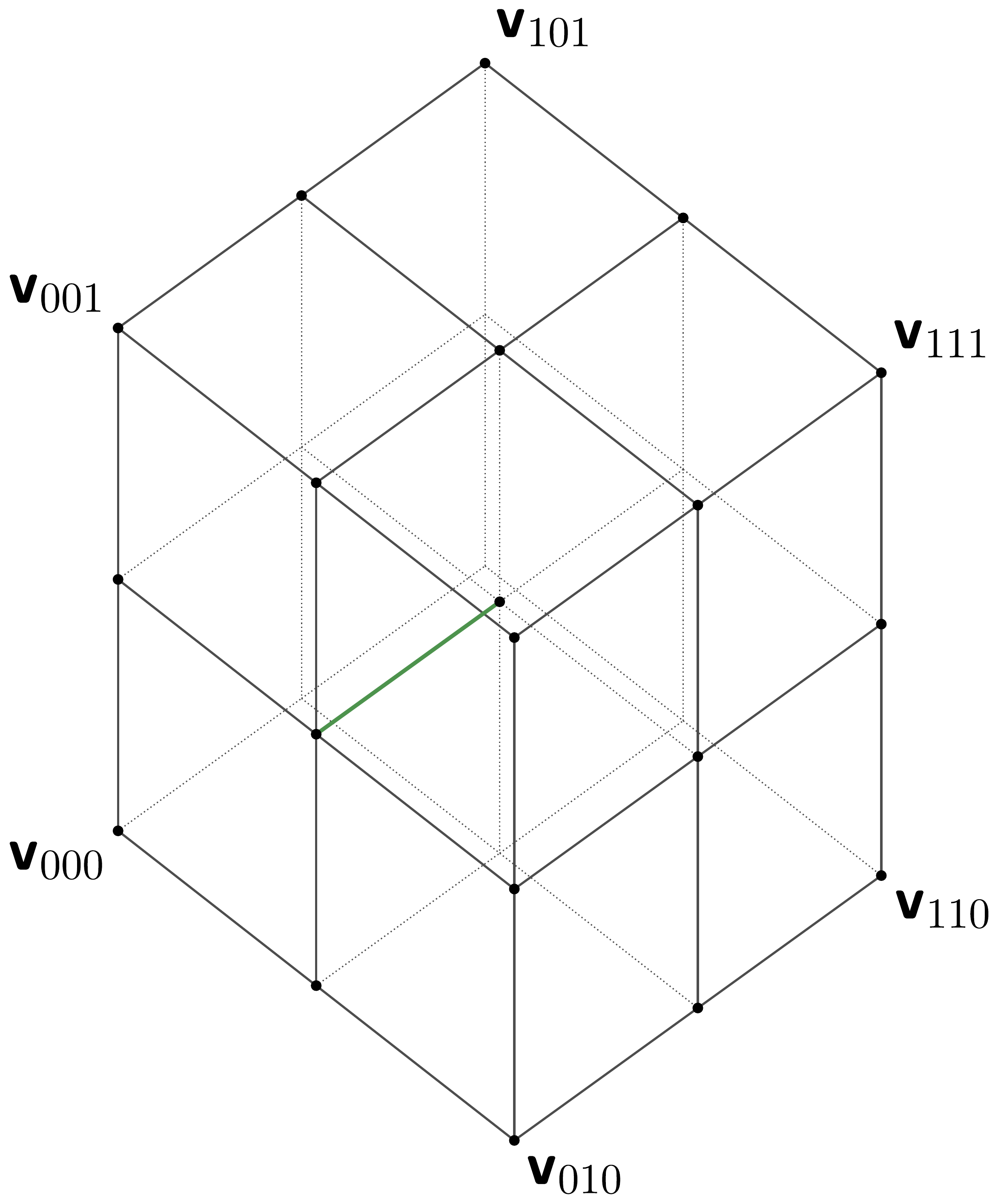} \hfill
	\caption{\it Illustration for Example~\ref{ex:3}: Three dimensional example of a refinement $\R(Q)$.}
	\label{fig:interpretation_locations}
\end{figure}
\begin{table}
\centering
\begin{tabular}{ll}\toprule
 location tuple $\myind{\ell}$ & location in refinement $\R(Q)$\\%& orientation \\
 \midrule
 $\myind{\ell} = (0,0)$ & left half in $x$-, left half in $y$-direction  \\%& parallel to $xy$-plane \\
 $\myind{\ell} = (1,0)$ & right half in $x$-, left half in $y$-direction \\%& parallel to $xy$-plane \\
 $\myind{\ell} = (0,1)$ & left half in $x$-, right half in $y$-direction    \\%& parallel to $xy$-plane \\
 $\myind{\ell} = (1,1)$ & right half in $x$-, right half in $y$-direction  \\%& parallel to $xy$-plane \\
 \bottomrule
\end{tabular}\medskip
\caption{\it Possible locations for interior 2-dimensional faces parallel to the $xy$-plane in the refinement of an element~$Q$, all with orientation tuple $\myind{a}=(1,2)$.}
\label{tb:loc}
\end{table}
\end{example}

For any $r\in\underline{d}_0$, let us introduce the set
\begin{align*}
 \N_r &:= \big\lbrace \myind{n} = (\myind{a},\myind{\ell}) : \myind{a}=(a_1,\ldots,a_r)\in D_r \text{ and } \myind{\ell}=(\ell_1,\ldots,\ell_r)\in\lbrace 0,1\rbrace^r \big\rbrace
\end{align*}
of all $r$-dimensional internal nodes in~$\R(Q)$; note that the cardinality of $\N_r$ is given by
\begin{align*}
 |\, \N_r| = 2^r\binom{d}{r}\, .
\end{align*}
Furthermore, we signify by $\N := \bigcup_{r\in \underline{d}_0} \N_r$ the collection of all internal nodes  of any dimension $r=0,\dots,d$ in $\R(Q)$. In addition, for any $\myind{n}=(\myind{a},\myind{\ell})\in\N$, we define
\begin{align*}
 \myind{I}(\myind{n}) 
 := \big\lbrace \myind{i} = (i_1,\ldots,i_d)\in \lbrace 0,1\rbrace^d : i_k = \ell_k \text{ for each } k\in A(\myind{a}) \big\rbrace
\end{align*}
to be the set of all multi-indices $\myind{i}\in\lbrace 0,1\rbrace^d$ corresponding to elements $T_{\myind{i}}\in\R(Q)$ that share the internal node $\myind{n}$. 

\begin{example}[$ \myind{I}(\myind{n})$ for $3$-dimensional example]\label{ex:4}
In the setting of Example~\ref{ex:3}, let us consider once more the 1-dimensional edge, which is given by the orientation and location tuples $\myind{a}=(1)$ and the $\myind{\ell}=(0)$. Then, the indices of all 3-dimensional elements $T_{\myind{i}}\in\R(Q)$ sharing the internal node $\myind{n}=\{((1),(0))\}\in\N_1$, which are highlighted in Figure~\ref{fig:interpretation_indexSet}, are collected in the set
$
 \myind{I}(\myind{n}) = \big\lbrace (\underline{0},0,0), (\underline{0},1,0), (\underline{0},1,1), (\underline{0},0,1) \big\rbrace.
$
\begin{figure}
	\centering
	\includegraphics[scale=0.175]{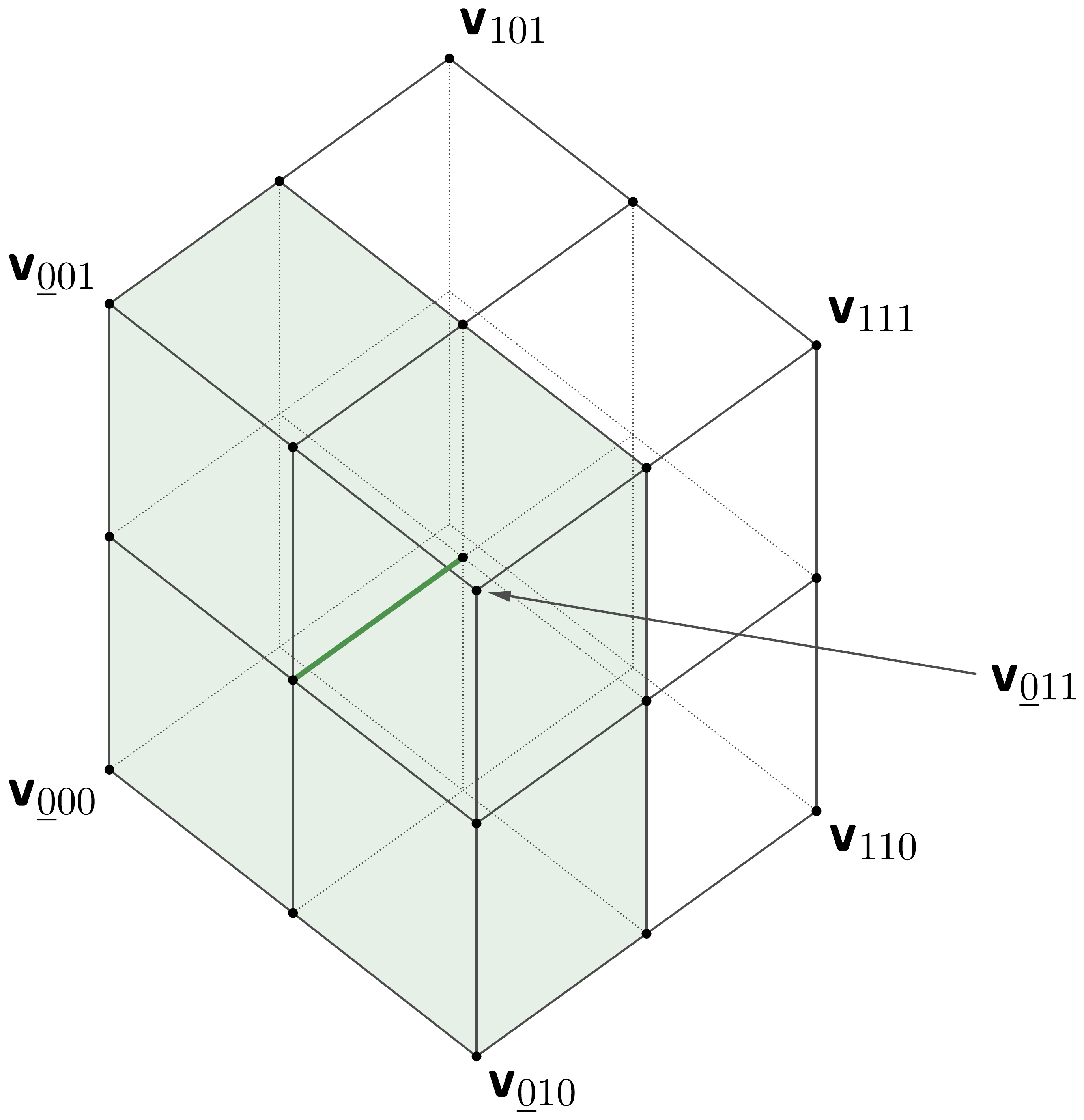}
	\caption{\it Illustration for Example~\ref{ex:4}: 1-dimensional edge in a 3D refinement.}
	\label{fig:interpretation_indexSet}
\end{figure}
\end{example}

% -----------------------------------------------------------------------------------------

\paragraph*{$hp$-enrichment functions on~$Q$}

Let $\myind{n}=(\myind{a},\myind{\ell})\in\N_r$ be an internal node of $\R(Q)$, for some $r\in\underline{d}_0$, and consider an associated polynomial distribution 
\begin{align*}
 \myind{p} = (p_1,\ldots,p_r)\in \big\lbrace \myind{p}\in\NN_0^r : p_k\geq 2 \text{ for } ,k\in\underline{r} \big\rbrace;
\end{align*}
here, the entry $p_k$ corresponds to the coordinate direction $k\in A(\myind{a})$ of $\myind{n}$. Then, for each $\myind{i}\in \myind{I}(\myind{n})$, we introduce a $d$-tuple $\myind{j}(\myind{i},\myind{p})=(j_1,\ldots,j_d)\in\NN_0^d$  component-wise by
\begin{align}\label{eq:defi_index_for_h_xi}
 j_k := \begin{cases}
  p_k, & \text{if }k\in A(\myind{a}), \\
  1-i_k, & \text{if } k\notin A(\myind{a}).
 \end{cases}
\end{align}
Let us now define functions $\xi_{\myind{n},\myind{p}}:\Omega\to\RR$, which are associated with the $r$-dimensional internal node $\myind{n}$, by
\begin{align}\label{eq:defi_xi_h_enrichment}
 \xi_{\myind{n},\myind{p}}(\myvec{x}) := \begin{cases}
  \zeta_{\myind{j}(\myind{i},\myind{p})}^{\myind{i}}(\myvec{x}), & \text{if } \myvec{x}\in T_{\myind{i}} ,\text{ for } \myind{i}\in\myind{I}(\myind{n}), \\
  0, & \text{if } \myvec{x} \in \Omega \setminus T(\myind{n}),
 \end{cases}
\end{align}
where
\begin{align}\label{eq:supp_xi}
 T(\myind{n}) := \bigcup_{\myind{i}\in\myind{I}(\myind{n})} T_{\myind{i}}
\end{align}
represents the support of the functions $\xi_{\myind{n},\myind{p}}$. As $hp$-enrichment functions we choose finitely many such functions. More precisely, for any internal node $\myind{n} = (\myind{a},\myind{\ell})\in\N_r$, $r\in\underline{d}_0$, we select a finite set
\begin{align*}%\label{eq:indexSet_hpEnr}
 \myind{P}(\myind{n})\subseteq \big\lbrace \myind{p}\in\NN_0^r : p_k\geq 2 \text{ for } k\in\underline{r} \big\rbrace,
\end{align*}
and let
\begin{align}\label{eq:general_h_enrichment}
 \mathfrak{E}_{hp} := \bigcup_{\myind{n}\in\N} \mathfrak{E}_{hp,\myind{n}}, 
  \qquad
 \mathfrak{E}_{hp,\myind{n}} := \big\lbrace \xi_{\myind{n},\myind{p}} : \myind{p}\in \myind{P}(\myind{n}) \big\rbrace.
\end{align}
For the single $0$-dimensional internal node of $\R(Q)$, represented by $\myind{n} = ((),())$, we note that $\myind{P}(\myind{n}) = \lbrace ()\rbrace$; hence, the only $hp$-enrichment function $\xi_{\myind{n},\myind{p}}$ associated with this node is given by
\begin{align*}
 \xi_{((),()),())}(\myvec{x}) = \begin{cases}
  \zeta_{\myvec{1}-\myind{i}}^{\myind{i}}(\myvec{x}), & \text{if } \myvec{x}\in T_{\myind{i}} ,\text{ for } \myind{i}\in\lbrace 0,1\rbrace^d, \\
  0, & \text{if } \myvec{x} \in \Omega \setminus T(\myind{n}),
 \end{cases}
\end{align*}
where $\myind{1} = (1,\ldots,1)\in\NN^d$.

\begin{proposition}\label{prop:properties_henrichmentF}
Any $\xi_{\myind{n},\myind{p}}\in\mathfrak{E}_{hp}$ is continuous in $\Omega$, and it holds $\operatorname{supp}(\xi_{\myind{n},\myind{p}}) = T(\myind{n})$.
\end{proposition}

\begin{proof}
Let $\xi_{\myind{n},\myind{p}}\in\mathfrak{E}_{hp,\myind{n}}$ with $\mathfrak{E}_{hp,\myind{n}}$ from \eqref{eq:general_h_enrichment}, be an arbitrary $hp$-refinement function, for some $\myind{n} =(\myind{a},\myind{\ell})\in\N_r$, $r\in\underline{d}_0$, and $\myind{p}\in\myind{P}(\myind{n})$. Let us denote by
$
 f_{\myind{n}} := \bigcap_{\myind{i}\in\myind{I}(\myind{n})} T_{\myind{i}}
$
the $r$-dimensional internal node of the refinement $\R(Q)$ characterized by $\myind{n}$. Noticing that
\begin{align*}
 \xi_{\myind{n},\myind{p} \, | \, T_{\myind{i}}} 
 = \zeta_{\myind{j}(\myind{i},\myind{p})}^{\myind{i}}
 = \myref{\psi}_{\myind{j}(\myind{i},\myind{p})} \circ F_{\myind{i}}^{-1}
 \qquad \forall \, \myind{i}\in\myind{I}(\myind{n}),
\end{align*}
and $\xi_{\myind{n},\myind{p} \, | \, T_{\myind{i}}} \equiv 0$ on elements $T_{\myind{i}}$ with $\myind{i}\in\lbrace 0,1\rbrace^d \setminus \myind{I}(\myind{n})$ as well as  on $\Omega\setminus Q$ by~\eqref{eq:defi_xi_h_enrichment}, we immediately obtain the continuity of the function $\xi_{\myind{n},\myind{p}}$ on the interior of the subelements $T_{\myind{i}}$, for $\myind{i}\in\myind{I}(\myind{n})$, and on $\Omega\setminus T(\myind{n})$. Furthermore, we have
$
 \operatorname{supp}(\xi_{\myind{n},\myind{p}}) = T(\myind{n}),
$
cf.~\eqref{eq:supp_xi}. Thus, in order to prove the continuity of $\xi_{\myind{n},\myind{p}}$ on $\Omega$,  we need to show that $\xi_{\myind{n},\myind{p}} \equiv 0$ on the boundary $\partial T(\myind{n})$ of the support $T(\myind{n})$, and that the functions $\zeta_{\myind{j}(\myind{i},\myind{p})}^{\myind{i}}$, for any $\myind{i}\in\myind{I}(\myind{n})$, coincide on the $r$-dimensional face $f_{\myind{n}}$. 

Since $\myind{p}=(p_1,\ldots,p_r)\in\myind{P}(\myind{n})$, we have $p_k\geq 2$ for any $k\in\underline{r}$. Hence, it holds that
$
  j_{k,\myind{i}} \geq 2 \text{ for } k\in A(\myind{a})
  $, and
  $ j_{k,\myind{i}}\in\lbrace 0,1\rbrace \text{ for } k\notin A(\myind{a})
$,
where the indices $(j_{1,\myind{i}},\ldots,j_{d,\myind{i}})=\myind{j}(\myind{i},\myind{p})\in\NN_0^d$ are defined in \eqref{eq:defi_index_for_h_xi}. To ensure that the functions
\begin{align}\label{eq_proof:prod_zetas}
 \zeta_{\myind{j}(\myind{i},\myind{p})}^{\myind{i}}(\myvec{x})
 = \prod_{k\in\underline{d}} \psi_{j_{k,\myind{i}}}(\myref{x}_k),
\qquad\myref{\myvec{x}} = (\myref{x}_1,\ldots,\myref{x}_d)^{\top} = F_{\myind{i}}^{-1}(\myvec{x}), \qquad
 \myind{i}\in\myind{I}(\myind{n}),
\end{align}
with the one dimensional functions $\psi_k$ from \eqref{eq:defi_1D_psi}, coincide for $\myvec{x}= (x_1,\ldots,x_d)^{\top}\in f_{\myind{n}}$, we first show that $\psi_{j_{k,\myind{i}}}(\widehat{x}_k)=1$ for any $k\notin A(\myind{a})$. Indeed, for $\myind{i}=(i_1,\ldots,i_d)\in\myind{I}(\myind{n})$ we have
\begin{align}\label{eq_proof:representation_node}
 f_{\myind{n}} = \big\lbrace F_{\myind{i}}(\myref{\myvec{x}}) : \myref{\myvec{x}} = (\myref{x}_1,\ldots,\myref{x}_d)^{\top}\in\myref{Q}, \text{ and } \myref{x}_k = 1 - 2\, i_k \text{ for } k\notin A(\myind{a})  \big\rbrace,
\end{align}
and because of
\begin{align*}
 j_{k,\myind{i}} = \begin{cases}
 1, & \text{if } i_k = 0, \\
 0, & \text{if } i_k = 1,
 \end{cases}
 \qquad k\notin A(\myind{a}),
\end{align*}
cf.~\eqref{eq:defi_index_for_h_xi}, we obtain
\begin{align*}
 \psi_{j_{k,\myind{i}}}(\myref{x}_k) = \begin{cases}
 \psi_1(1) = 1, & \text{if } i_k = 0, \\
 \psi_0(-1) = 1, & \text{if } i_k = 1,
 \end{cases} \qquad k\notin A(\myind{a}),
\end{align*}
owing to~\eqref{eq:defi_1D_psi} and \eqref{eq_proof:representation_node}. Next, for $k\in A(\myind{a})$ it holds that
$
 \psi_{j_{k,\myind{i}}}(\myref{x}_k) = \psi_{p_k}(\myref{x}_k)$, cf.~\eqref{eq:defi_index_for_h_xi}. Thus,
\begin{align*}
 \zeta_{\myind{j}(\myind{i},\myind{p})}^{\myind{i}}(\myvec{x})
 = \prod_{k\in A(\myind{a})} \psi_{p_k}(\myref{x}_k),
 \qquad
 \myind{i}\in\myind{I}(\myind{n}),
\end{align*}
which shows the continuity of the enrichment function $\xi_{\myind{n},\myind{p}}$ in the interior of the support $T(\myind{n})$. Finally, we observe that  $\psi_0(1) = \psi_1(-1) =0$, and
$
 \psi_j(\pm 1) = 0,
$
for any $j\in\NN$ with $j\geq 2$. Then, exploiting that
\begin{align*}
 \partial T(\myind{n}) \setminus f_{\myind{n}}
 = \bigcup_{\myind{i}\in\myind{I}(\myind{n})} \big\lbrace F_{\myind{i}}(\myref{\myvec{x}}) : \myref{\myvec{x}}\in\myref{Q} \text{ and } \exists \, k\in A(\myind{a}): \myref{x}_k = \pm 1 \; \vee \; \exists \, k \notin A(\myind{a}): \myref{x}_k = 2 \, i_k - 1 \big\rbrace,
\end{align*}
and applying~\eqref{eq_proof:prod_zetas}, we arrive at $\zeta_{\myind{j}(\myind{i},\myind{p})}^{\myind{i}} (\myvec{x}) = 0$, for any $\myvec{x}\in\partial T(\myind{n})\setminus f_{\myind{n}}$, $\myind{i}\in\myind{I}(\myind{n})$,  This completes the argument.
\end{proof}

\begin{remark}[$\myind{P}(\myind{n})$ for uniform polynomial degrees]
Again, there are various possibilities to specify $hp$-enrichment functions. Note that in the definition of $\mathfrak{E}_{hp}$, the maximal polynomial degree can differ on each node $\myind{n}\in\N$, however, we could also consider the special case of a uniform polynomial degree distribution $p_{\text{unif}}\in\lbrace 2,\ldots, p_{\max} \rbrace$ on each element $T_{\myind{i}}\in\R(Q)$, $\myind{i}\in\lbrace 0,1\rbrace^d$, obtained by choosing $\myind{P}(\myind{n}) = \lbrace 2,\ldots, p_{\text{unif}}\rbrace^r$ for any $\myind{n}\in\N_r$, $r\in\underline{d}$; see Example~\ref{ex:5} below for further comments on this particular choice.
\end{remark}

% -----------------------------------------------------------------------------------------
%    Subsection: Representation matrices
% -----------------------------------------------------------------------------------------

\subsection{Representation matrices}\label{subsec:representation_matrices}

Owing to the previous Propositions~\ref{prop:properties_penrichmentF} and \ref{prop:properties_henrichmentF} the sum over all elements $T\in\T$ in \eqref{eq:assembling_1} reduces to a sum over all elements $T_{\myind{i}}$ in the refinement $\R(Q)$. In particular, the representation matrices $\mymatrix{C}_{\myind{i}}$ and $\mymatrix{D}_{\myind{i}}$, see the notation in \S\ref{sec:1element}, need to be specified only for all $T_{\myind{i}}\in\R(Q)$, $\myind{i}\in\lbrace 0,1\rbrace^d$. Recalling Proposition~\ref{prop:assembling}, we remark that these (local) quantities, in turn, are required to assemble the matrix $\mymatrix{A}$ and the vectors $\myvec{b},\myvec{c}$ occurring in the linear system \eqref{eq:linsys}.

In order to compute the matrices $\mymatrix{C}_{\myind{i}}$ and $\mymatrix{D}_{\myind{i}}$ on each element $T_{\myind{i}}\in\R(Q)$, $\myind{i}\in\lbrace 0,1\rbrace^d$, we employ the constraint coefficients technique from \S\ref{subsec:constraint_coeff}. To this end, let $\mathfrak{E}_{Q} := \lbrace\xi_1,\ldots,\xi_L\rbrace$ be a set of $p$- or $hp$-enrichment functions on $Q$. On each $T_{\myind{i}}\in\R(Q)$, $\myind{i}\in\lbrace 0,1\rbrace^d$, for $\myind{j}\in \lbrace0,\ldots,p_{\max}\rbrace^d$, we consider the functions $\zeta_{\myind{j}}^{\myind{i}}$ from~\eqref{eq:zetas}, which we enumerate by the bijective map $\iota : \lbrace 0,\ldots, p_{\max}\rbrace^d \to \underline{M}$, with $M := (p_{\max} + 1)^d$, given by
\begin{align}\label{eq:enumeration_01}
 \iota(\myind{j}) := 1 + \sum_{k\in\underline{d}} (p_{\max}+1)^{k-1} \, j_k,
  \qquad \myind{j}=(j_1,\ldots,j_d),
\end{align}
i.e. we apply the renumbering $\zeta_{\iota(\myind{j})}^{\myind{i}} := \zeta_{\myind{j}}^{\myind{i}}$, for
$\myind{j}\in\lbrace 0,\ldots,p_{\max} \rbrace^d.$
Note that the functions $\lbrace \zeta_1^{\myind{i}},\ldots,\zeta_M^{\myind{i}}\rbrace$, extended by zero to $\Omega$, are discontinuous and represent a basis of a (local) discontinuous Galerkin space on $\R(Q)$. Given a basis $\lbrace \phi_1, \ldots, \phi_N\rbrace$ of the finite element space $\WW$, cf.~\eqref{eq:FE_space}, recall that the components of the matrix $\mymatrix{C}_{\myind{i}} = (\, c_{kl}^{\myind{i}} \,)\in\RR^{N\times M}$, for $\myind{i}\in\lbrace 0,1\rbrace^d$, cf.~\eqref{eq:rep_phi_xi_01}, are determined by 
\begin{align*}
 \phi_{k \,| \, T_{\myind{i}}} &= \sum_{l\in\underline{M}} c_{kl}^{\myind{i}} \, \zeta_l^{\myind{i}},
  \qquad k\in\underline{N}, %\label{eq:rep_phi_xi_neu_01}
\end{align*}
and the components of $\mymatrix{D}_{\myind{i}}=( \, d_{kl}^{\myind{i}} \,)\in\RR^{L\times M}$, $\myind{i}\in\lbrace 0,1\rbrace^d$, cf.~\eqref{eq:rep_phi_xi_02}, are given by 
\begin{align}
  \xi_{k \,| \, T_{\myind{i}}} &= \sum_{l\in\underline{M}} d_{kl}^{\myind{i}} \, \zeta_l^{\myind{i}},
 \qquad k\in\underline{L}. \label{eq:rep_phi_xi_neu_02}
\end{align}

% -----------------------------------------------------------------------------------------
%    Subsubsection: Computation of the matrices C_i

\subsubsection{Computation of the matrices $\mymatrix{C}_{\myind{i}}$}

On an element $Q\in\Q$,  for any $\myind{j}\in\lbrace 0,\ldots,p_{\max}\rbrace^d$, we introduce the functions $\psi_1^{Q}, \ldots, \psi_M^{Q}$ by
\begin{align*}%\label{eq:defi_psi_on_Q}
 \psi_{\iota(\myind{j})}^{Q}(\myvec{x}) := \myref{\psi}_{\myind{j}} \circ F_{Q}^{-1}(\myvec{x}),
  \qquad \myvec{x}\in Q,
\end{align*}
with the bijective map $\iota$ from \eqref{eq:enumeration_01}. Furthermore, let $\mymatrix{C}_{Q} = \big(\, c_{kl}^{Q} \,\big)\in\RR^{N\times M}$ denote the representation matrix, for which it holds
\begin{align*}
 \phi_{k \, | \, Q} = \sum_{l\in\underline{M}} c_{kl}^{Q} \, \psi_l^{Q},
  \qquad k\in\underline{N}.
\end{align*}
Finally, for any $\myind{i}\in\lbrace 0,1\rbrace^d$, let us define  the matrix $\mymatrix{B}_{\myind{i}} = (\, b_{kl}^{\myind{i}} \,)\in\RR^{M\times M}$ component-wise by
$
 b_{kl}^{\myind{i}} := b_{\myind{k}\myind{l}}^{T_{\myind{i}}},
$
where $k = \iota(\myind{k})$, $l = \iota(\myind{l})$, for any $\myind{k},\myind{l}\in\lbrace 0,\ldots,p_{\max}\rbrace^d$, and $b_{\myind{k}\myind{l}}^{T_{\myind{i}}}$ are the constrained coefficients from \eqref{eq:representation_restriction}.

\begin{proposition}
For any $\myind{i}\in\lbrace 0,1\rbrace^d$ it holds that $\mymatrix{C}_{\myind{i}} =  \mymatrix{C}_{Q} \, \mymatrix{B}_{\myind{i}}$.
\end{proposition}

\begin{proof}
For any $k\in\underline{N}$, we have
\begin{align*}
 \phi_{k \, | \, Q} = \sum_{l\in\underline{M}} c_{kl}^{Q} \, \psi_l^Q
 = \sum_{l\in\underline{M}} c_{kl}^{Q} \,\left( \myref{\psi}_l \circ F_{Q}^{-1}\right).
\end{align*} 
Therefore, we infer that
\begin{align*}
 \phi_{k \,| \, T_{\myind{i}}} \circ F_{Q} = \sum_{l\in\underline{M}} c_{kl}^{Q} \, \myref{\psi}_{l \, |\myref{T}_{\myind{i}}},
 \qquad \myind{i}\in\lbrace 0,1\rbrace^d.
\end{align*}
Denoting by $\myref{F}_{\myind{i}} : \myref{Q}\to \myref{T}_{\myind{i}}$ the bijective map from $\myref{Q}$ to the sub-hexahedron $\myref{T}_{\myind{i}}$, and noting that $F_{\myind{i}} = F_{Q} \circ \myref{F}_{\myind{i}}$,  the restrictions to $\myref{T}_{\myind{i}}$ can be represented by
\begin{align*}
 \myref{\psi}_{l \, |\myref{T}_{\myind{i}}}
 = \sum_{r\in\underline{M}} b_{\myind{l}\myind{r}}^{T_{\myind{i}}} \, \left(\myref{\psi}_{r} \circ \myref{F}_{\myind{i}}^{-1}\right),
\end{align*}
for any $\myind{i}\in\lbrace 0,1\rbrace^d$, cf.~\eqref{eq:representation_restriction}, where $\myind{l} = \iota^{-1}(l)$ and $\myind{r} = \iota^{-1}(r)$. Thus,
\begin{align*}
\phi_{k \,| \, T_{\myind{i}}} 
 = \sum_{l\in\underline{M}} c_{kl}^{Q} \sum_{r\in\underline{M}} b_{\myind{l}\myind{r}}^{T_{\myind{i}}} \, \left(\myref{\psi}_r \circ \myref{F}_{\myind{i}}^{-1} \circ F_{Q}^{-1}\right)
 = \sum_{l\in\underline{M}} c_{kl}^{Q} \sum_{r\in\underline{M}} b_{\myind{l}\myind{r}}^{T_{\myind{i}}} \, \left(\myref{\psi}_r \circ F_{\myind{i}}^{-1}\right).
\end{align*}
Since $\zeta_r^{\myind{i}} = \myref{\psi}_r \circ F_{\myind{i}}^{-1}$, we obtain
\vspace{-0.6cm}\\
\begin{align*}
 \phi_{k \,| \, T_{\myind{i}}} 
 = \sum_{l\in\underline{M}} c_{kl}^{Q} \sum_{r\in\underline{M}} b_{\myind{l}\myind{r}}^{T_{\myind{i}}} \, \zeta_r^{\myind{i}}
 = \sum_{r\in\underline{M}} \Bigg( \sum_{l\in\underline{M}} c_{kl}^{Q} \, b_{\myind{l}\myind{r}}^{T_{\myind{i}}} \Bigg) \zeta_r^{\myind{i}},
\end{align*}
from which we deduce
\begin{align*}
 c_{kr}^{\myind{i}} 
 = \sum_{l\in\underline{M}} c_{kl}^{Q} \, b_{\myind{l}\myind{r}}^{T_{\myind{i}}},
\end{align*}
for any $k\in\underline{N}$ and $r\in\underline{M}$. Thus, $\mymatrix{C}_{\myind{i}} = \mymatrix{C}_{Q} \, \mymatrix{B}_{\myind{i}}$ for any $\myind{i}\in\lbrace 0,1\rbrace^d$.
\end{proof}

% -----------------------------------------------------------------------------------------
%    Subsubsection: Representation matrices for p-enrichment on Q

\subsubsection{Computation of $\mymatrix{D}_{\myind{i}}$ for a $p$-enrichment on $Q$}

In this case the enrichment functions $\mathfrak{E}_p = \lbrace \xi_1,\ldots,\xi_L\rbrace$ represent a certain collection of the functions $\xi_{\myind{j}}$ in~\eqref{eq:defi_xi_p_enrichment}, with $\myind{j}$ from a multi-index set
\begin{align*}
 \myind{J}_d \subseteq \big\lbrace \myind{j}\in\NN_0^d : j_k \geq 2 \text{ for } k\in\underline{d} \big\rbrace, \qquad
 |\myind{J}_d| = L.
\end{align*}
As in the previous section, we apply a renumbering $\xi_{\kappa(\myind{j})} := \xi_{\myind{j}}$ in terms of a bijective map $\kappa : \myind{J}_d \to \underline{L}$; for instance, if we choose $\myind{J}_d := \lbrace 2,\ldots,p_{\max}\rbrace^d$, then the resulting $p$-enrichment functions $\xi_1,\ldots,\xi_L$, with $L = (p_{\max}-1)^d$, can be enumerated by
\begin{align*}
 \kappa(\myind{j}) := 1 + \sum_{k\in\underline{d}} (p_{\max}-1)^{k-1} (j_k-2), 
  \qquad \myind{j} = (j_1,\ldots,j_d).
\end{align*}

\begin{proposition}
For any $\myind{i}\in\lbrace 0,1\rbrace^d$, the matrix $\mymatrix{D}_{\myind{i}}\in\RR^{L\times M}$ is given component-wise by
$
 d_{kl}^{\myind{i}} = b_{\myind{k}\myind{l}}^{T_{\myind{i}}},
$
for $k = \kappa(\myind{k})$ and $l = \iota(\myind{l})$.
\end{proposition}

\begin{proof}
By the definition \eqref{eq:defi_xi_p_enrichment} of $p$-enrichment functions we have
$
 \xi_{k \, | \, T_{\myind{i}}}
 = \myref{\psi}_{k \, | \, \myref{T}_{\myind{i}}} \circ F_{Q},
$
for $\myind{i}\in\lbrace 0,1\rbrace^d$ and any $k\in \underline{L}$. Hence, using the bijective map $\myref{F}_{\myind{i}}:\myref{Q} \to \myref{T}_{\myind{i}}$ we obtain 
\begin{align*}
 \myref{\psi}_{k \, | \, \myref{T}_{\myind{i}}} 
 = \sum_{l\in\underline{M}} d_{kl}^{\myind{i}} \, \left(\zeta_l^{\myind{i}} \circ F_{Q}^{-1}\right)
 = \sum_{l\in\underline{M}} d_{kl}^{\myind{i}} \, \left(\myref{\psi}_l \circ F_{\myind{i}}^{-1} \circ F_{Q}^{-1}\right)
 = \sum_{l\in\underline{M}} d_{kl}^{\myind{i}} \, \left(\myref{\psi}_l \circ \myref{F}_{\myind{i}}^{-1}\right)
\end{align*}
by the representation \eqref{eq:rep_phi_xi_neu_02}. Hence, by virtue of~\eqref{eq:representation_restriction}, we arrive at $d_{kl}^{\myind{i}} = b_{\myind{k}\myind{l}}^{T_{\myind{i}}}$ for $k = \kappa(\myind{k})$ and $l = \iota(\myind{l})$.
\end{proof}

% -----------------------------------------------------------------------------------------
%    Subsubsection: Representation matrices for p-enrichment on Q

\subsubsection{Computation of $\mymatrix{D}_{\myind{i}}$ for an $hp$-refinement on $Q$}

We consider $hp$-enrichment functions $\mathfrak{E}_{hp} = \lbrace \xi_1,\ldots,\xi_L\rbrace$ from a certain selection of functions $\xi_{\myind{n},\myind{p}}$ in~\eqref{eq:defi_xi_h_enrichment}. Here, each $\myind{n} = (\myind{a},\myind{\ell})\in\N$ corresponds to an internal node of the refinement $\R(Q)$, which is characterized by the tuples $\myind{a}\in D_r$ and $\myind{\ell}\in\lbrace 0,1\rbrace^r$, for some $r\in\underline{d}_0$, cf.~\S\ref{subsec:hp_enrichments}. Moreover, $\myind{p}\in\myind{P}(\myind{n})$ represents the polynomial distribution for the directions $k\in A(\myind{a})$, where
\begin{align*}
 \myind{P}(\myind{n})\subseteq \big\lbrace \myind{p}\in\NN_0^r : p_k\geq 2 \text{ for } k\in\underline{r} \big\rbrace, \qquad
 \text{with} \quad
 |\myind{P}(\myind{n})| < \infty.
\end{align*}
Given a polynomial degree distribution $(p_{\myind{i}})_{\myind{i}\in\lbrace 0,1\rbrace^d}$ on the subelements $T_{\myind{i}}\in\R(Q)$,  with $p_{\myind{i}}\in\lbrace 1,\ldots,p_{\max}\rbrace$, we can associate a \emph{nodal polynomial degree} $p_{\myind{n}}$ to any node $\myind{n}\in\N_r$, with $r\in\underline{d}$, by means of the minimum-rule
\begin{align*}%\label{eq:minRule_polyDegs}
 p_{\myind{n}} := \min_{\myind{i}\in T(\myind{n})} p_{\myind{i}}
\end{align*}
In that case, for any node $\myind{n}\in\N_r$, we set $\myind{P}(\myind{n}) = \lbrace 2,\ldots,p_{\myind{n}}\rbrace^r$ (with the convention $\myind{P}(((),())) = \lbrace ()\rbrace$ for $r=0$), and we can enumerate the functions $\xi_{\myind{n},\myind{p}}$ in terms of a bijective map $\nu : \big\lbrace (\myind{n},\myind{p}) : \myind{n}\in\N, \, \myind{p}\in\myind{P}(\myind{n}) \big\rbrace \to \underline{L}$. 

\begin{proposition}
For any $\myind{i}\in\lbrace 0,1\rbrace^d$ the matrix $\mymatrix{D}_{\myind{i}}\in\RR^{L\times M}$ is given component-wise by
\begin{align*}
 d_{kl}^{\myind{i}} = \begin{cases}
  1, & \text{if } \xi_{\myind{n},\myind{p}} = \zeta_{\myind{j}}^{\myind{i}}, \\
  0, & \text{otherwise},
 \end{cases}
\end{align*}
for $k = \nu(\myind{n},\myind{p})$ and $l = \iota(\myind{j})$.
\end{proposition}

\begin{proof}
By the definition \eqref{eq:defi_xi_h_enrichment} of $hp$-enrichment functions, we have
$
 \xi_{\myind{n},\myind{p} \, | \, T_{\myind{i}}}
 = \zeta_{\myind{j}(\myind{i},\myind{p})}^{\myind{i}},
$
for any $\myind{i}\in \myind{I}(\myind{n})$, where $\myind{j}(\myind{i},\myind{p})\in\NN_0^d$ is defined in \eqref{eq:defi_index_for_h_xi}. Since the components of $\mymatrix{D}_{\myind{i}}$ are determined by
\begin{align*}
 \xi_{k \,| \, T_{\myind{i}}} = \sum_{l\in\underline{M}} d_{kl}^{\myind{i}} \, \zeta_l^{\myind{i}},
 \qquad k = \nu(\myind{n},\myind{p}),
\end{align*}
we immediately obtain $d_{kl}^{\myind{i}} = 1$ if $l = \iota(\myind{j}(\myind{i},\myind{p}))$, and $d_{kl}^{\myind{i}} = 0$ otherwise. In addition, for any $\myind{i}\in\lbrace 0,1\rbrace^d \setminus \myind{I}(\myind{n})$, we have $\xi_{\myind{n},\myind{p} \, | \, T_{\myind{i}}} \equiv 0$, wherefore the linear independence of the functions $\zeta_1^{\myind{i}},\ldots,\zeta_M^{\myind{i}}$ yields $d_{kl}^{\myind{i}}=0$ for any $l\in\underline{M}$.
\end{proof}

\begin{example}[Enumeration $\nu$ for uniform polynomial degrees]\label{ex:5}
Let us assign the same polynomial degree $p_{\text{unif}} \in \lbrace 2,\ldots, p_{\max}\rbrace$ to all elements $T_{\myind{i}}$ of the refinement $\R(Q)$, $\myind{i}\in\lbrace 0,1\rbrace^d$, i.e.~$
 \myind{P}(\myind{n}) = \lbrace 2,\ldots, p_{\text{unif}}\rbrace^r,
$
for any $r$-dimensional internal node $\myind{n}\in\N_r$, $r\in\underline{d}_0$. In this case, we observe that there are $(p_{\text{unif}} - 1)^r$ $hp$-enrichment functions on $Q$ that can be associated with $\myind{n}$. Hence, the total number of $hp$-refinement functions, which are associated with $r$-dimensional internal nodes, is given by
\begin{align*}
 L_r := \binom{d}{r} \, 2^r \, (p_{\text{unif}} - 1)^r, \qquad r\in\underline{d}_0.
\end{align*}
Now let us enumerate the $r$-dimensional internal nodes $\myind{n}=(\myind{a},\myind{p})\in\N_r$, $r\in\underline{d}_0$, by the bijective map $\nu_r : \N_r \to \big\lbrace 1,\ldots, \binom{d}{r} \, 2^r \big\rbrace$ given by
\begin{align*}
 \nu_r(\myind{n}) := 1 +  2^r\sum_{k\in\underline{r}} (a_k-k)   + \sum_{k\in\underline{r}} 2^{k-2} \, (\ell_k + 1),
 \qquad \myind{a}=(a_1,\ldots,a_r), \quad \myind{p}=(p_1,\ldots,p_r).
\end{align*}
Then, we can enumerate all $hp$-enrichment functions $\xi_{\myind{n},\myind{p}}$, for $\myind{n}\in\N$ and $\myind{p}\in\myind{P}(\myind{n})$, by the bijective map $\nu : \big\lbrace (\myind{n},\myind{p}) : \myind{n}\in\N, \, \myind{p}\in\myind{P}(\myind{n}) \big\rbrace \to \underline{L}$, with $L := L_0 + \cdots + L_d$, defined by
\begin{align*}
 \nu(\myind{n},\myind{p}) := 1 + \sum_{k\in\underline{r}} L_{k-1} + (\nu_r(\myind{n})-1) \, (p_{\text{unif}} - 1)^r + \sum_{k\in\underline{r}} (p_{\text{unif}} - 1)^{k-1} \, (p_k - 2),
\end{align*}
where we let $L_{-1} := 0$.
\end{example}

% -----------------------------------------------------------------------------------------
%    Section: An hp-adaptive procedure
% -----------------------------------------------------------------------------------------

\section{$hp$-adaptivity based on locally predicted error reductions}\label{sec:adaptive_alg}

In this section, we will exploit our abstract results in~\S\ref{subsec:low_dim_enrichments} for the purpose of devising a new adaptive procedure for $hp$-type finite element discretizations. Our basic idea to $hp$-refine a given $hp$-finite element space~$\WW_\text{given}$ consists of three essential steps: 
\begin{enumerate}[Step 1.]

\item Firstly, our algorithm aims to \emph{predict} the potential contribution to the (global)  energy error reduction from each individual element~$Q$ in the given $hp$-space $\WW_{\text{given}}$. To this end, for every element~$Q$, with an associated local space $\WWloc$, cf.~\eqref{eq:Wloc}, we apply various \emph{local $p$-enrichment} or \emph{$hp$-replacement spaces}~$\YY$ , cf.~\eqref{eq:Y}, as follows:\medskip

\begin{enumerate}[(hp)]

\item[(p)] In the case of $p$-enrichments, we choose finitely many sets $\myind{J}_{d,i}$ in order to define suitable collections of $p$-enrichment functions $\mathfrak{E}_{p,i} := \lbrace \xi^Q_{\myind{j}} : \myind{j}\in\myind{J}_{d,i} \rbrace$, cf.~\eqref{eq:penrich}, and let $\YY_{p,i}:=\operatorname{span}\{\ut\}+\operatorname{span}\mathfrak{E}_{p,i}$ be the associated $p$-enrichment spaces. We then compute the respective predicted error reductions $\Delta e_{p,i}$ by means of the formula~\eqref{eq:errid} (or equivalently \eqref{eq:formula}). \medskip

\item In the case of $hp$-refinements, for any node $\myind{n}$ of a refinement $\R(Q)$ of $Q$, we choose finitely many sets $\myind{P}_{i}(\myind{n})$  to define collections of (nodal) $hp$-enrichment functions $\mathfrak{E}_{hp,\myind{n}}^{i}:=\big\lbrace \xi^Q_{\myind{n},\myind{p}} : \myind{p}\in \myind{P}_{i}(\myind{n}) \big\rbrace$, cf.~\eqref{eq:general_h_enrichment}, and let 
\begin{align*}
 \YY_{hp,i}:=\operatorname{span}\{\ut\}+\operatorname{span}\mathfrak{E}_{hp, i},\qquad\text{with}\qquad \mathfrak{E}_{hp, i}:= \bigcup_{\myind{n}} \mathfrak{E}_{hp,\myind{n}}^{i},
\end{align*}
be the associated $hp$-replacement spaces. We then compute the corresponding predicted error reductions $\Delta e_{hp,i}$ by means of the formula~\eqref{eq:errid} (or equivalently \eqref{eq:formula}). \medskip

\end{enumerate}
From all these choices of $p$-enrichments~(p) and $hp$-refinements~(hp), we select  the one that features the maximal error reduction (signified by $\Delta e_{\max}^Q$); this will be referred to as the \emph{optimal (local) $hp$-refinement} on~$Q$. Preferably, following the approach proposed in~\cite{houston:2016}, we select the local $p$-enrichments and $hp$-refinements in a \emph{competitive} way, i.e.~with a comparable number of local degrees of freedom; we will discuss this idea in more detail in the context of the $1$-dimensional numerical examples in \S\ref{sec:numerical_examples} below.\medskip

\item Subsequently, we \emph{mark} all elements in the global $hp$-finite element space~$\WW_{\text{given}}$, from which the most substantial error reductions, as identified in Step~1. for each element, can be expected. This step can be accomplished, for instance, with the aid of a suitable marking strategy, such as D\"{o}rfler's criterion~\cite{dorfler:1996}.\medskip

\item Finally, a new \emph{enriched} $hp$-finite element space $\WW_\text{new}$ is constructed based on choosing the optimal (local) $hp$-refinement space, cf. Step~1., for each of the marked elements.\\

\end{enumerate}
Schematically, the proposed $hp$-adaptive procedure has the following structure:\\
\begin{center}	
\includegraphics[scale=0.4]{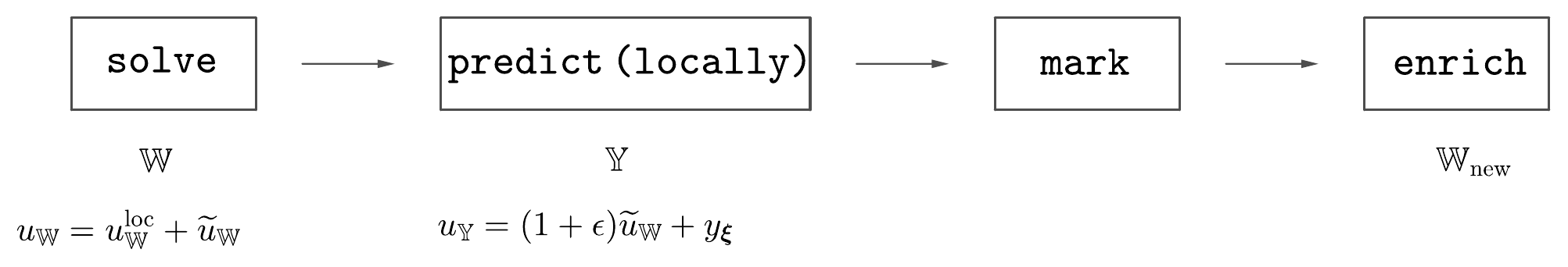}
\end{center}

\begin{remark}[Computational aspects of the prediction step]
We emphasize that the prediction step (Step~1.) is trivially parallelizable, and hence, inexpensive from a computational point of view. In addition, as all enrichment functions for the different $p$-enrichments and $hp$-refinements on $Q$ can be expressed by the functions $\zeta^{\myind{i}}_1,\ldots,\zeta^{\myind{i}}_M$, cf.~\eqref{eq:zetas}, the matrices $\mymatrix{A}_{\myind{i}}$, $\mymatrix{C}_{\myind{i}}$ as well as the vector $\myvec{b}_{\myind{i}}$ from~\S\ref{subsec:representation_matrices}, which are employed in the computation of the corresponding predicted local error reductions, need to be computed \emph{once only} on the subelements $T_{\myind{i}}$ of the local refinement $\R(Q)$ for \emph{all} different $p$-enrichments and $hp$-refinements applied on the subelements $T_{\myind{i}}$. 
\end{remark}

In Algorithm~1 we outline the technical details for the individual steps in the $hp$-adaptive refinement approach sketched above. We denote by $\WW(\Q,\myind{p}_{\Q})$ the $hp$-finite element space associated with the mesh $\Q$, and by $\myind{p}_{\Q}$ a corresponding polynomial degree distribution, cf.~\eqref{eq:FE_space}. Note that the \emph{solving step}, the \emph{prediction step} (exploiting $p$-enrichments and $hp$-refinements), the \emph{marking step} and the \emph{enrichment step} are implemented in the lines {\footnotesize{2}}, {\footnotesize{3-15}}, {\footnotesize{16}} and {\footnotesize{17}} in Algorithm~1, respectively.\\

\begin{remark}[Possible specifications of $hp$-adaptive procedure]~

\begin{enumerate}[(i)]

\item Evidently, Algorithm~1 can be turned into a pure $h$-adaptive procedure (by only considering $hp$-refinements that  inherit the \emph{fixed low-order} polynomial degree to any subelements), or into a pure $p$-adaptive procedure (whereby no element refinements are taken into account).

\item In the context of finite element discretizations for which the global polynomial degree is restricted to a small positive number $p_{\max}\ge 2$ (e.g. for quadratic elements), it is possible to first apply an $h$-adaptive mesh refinement procedure, thereby yielding a locally refined mesh~$\Q$, and then to employ $p$-enrichments based on Algorithm~1 in order to determine an effective polynomial distribution $\myind{p}_{\Q}$ with $\max\myind{p}_{\Q}\le p_{\max}$.

\end{enumerate}
\end{remark}

\begin{algorithm}[ht!]
{\textbf{Algorithm~1.} $hp$-adaptive procedure}
\vspace{0.1cm}
\hrule
\vspace{0.1cm}
\begin{algorithmic}[1]
\State{Choose an initial mesh $\Q_0$ on the computational domain $\Omega$ and a starting polynomial degree distribution~$\myind{p}_{\Q_0}$. Set $n:=0$.}
\State{Solve the weak formulation \eqref{eq:solW} for $u_{\WW}\in\WW(\Q_n, \myind{p}_{\Q_n})$.}\Comment{\textbf{solving step}}

\For {each element $Q\in\Q_n$}\Comment{\textbf{prediction step}}

\myState{Construct the (locally supported) subspace $\WWloc_Q\subset\WW(\Q_n, \myind{p}_{\Q_n})$, and decompose the current solution $u_\WW=\ut+\uloc$ according to \eqref{eq:dw},
where $\uloc := \Ploc_Q u_\WW\in\WWloc_Q$ is defined via the linear projection operator $\Ploc_Q : \WW(\Q_n, \myind{p}_{\Q_n})\to\WWloc_Q$ from~\eqref{eq:linear_projOp}.}
\For {finitely many different $p$-enrichments on $Q$}\Comment{$p$-enrichments}
\myStateDouble{Compute the corresponding predicted error reductions $\Delta e_{p,i}$ as outlined in Step.~1~(p) above.}
\EndFor

\myState{Set $\Delta e_{p,\max}^Q := \max \{\Delta e_{p,i}\}$ to be the maximal locally predicted error reduction for the $p$-enrichments on the element~$Q$.}
\myState{Construct the local refinement $\R(Q)$.\Comment{$hp$-refinement}}
\For {finitely many different $hp$-refinements on $Q$}
\myStateDouble{Compute the corresponding predicted error reductions $\Delta e_{hp,i}$ as outlined in Step.~1~(hp) above, and store the polynomial degree distribution $(\myind{p}_{T,i})_{T\in\R(Q)}$ for the subelements $T\in\R(Q)$.}
\EndFor

\myState{Set $\Delta e_{hp,\max}^Q := \max \{\Delta e_{hp,i}\}$ to be the maximal locally predicted error reduction for the $hp$-refinements on~$Q$.}
\myState{Determine $\Delta e_{\max}^{Q} := \max\big\lbrace \Delta e_{p,\max}^Q, \Delta e_{hp,\max}^Q \big\rbrace$ to be the maximal locally predicted error reduction.}

\EndFor
\State{Mark a subset $\mathcal{E}_n\subseteq \Q_n$ of elements in the mesh $\Q_n$ to be flagged for enrichment.}\Comment{\textbf{marking step}}
\State{For each element $Q\in\mathcal{E}_n$ perform a $p$-enrichment or an $hp$-refinement according to which enrichment leads to the maximal error reduction in {\footnotesize 14}. This results in a refined mesh $\Q_{n+1}$ with a corresponding polynomial degree distribution $\myind{p}_{\Q_{n+1}}$.}\Comment{\textbf{enrichment step}}

\State{Update $n\leftarrow n+1$, and start over in line {\footnotesize 2}.}
\State{After sufficiently many iterations output the final solution $u_{\WW}\in\WW(\Q_n, \myind{p}_{\Q_n})$.}
\end{algorithmic}
\end{algorithm}

% -----------------------------------------------------------------------------------------
%    Section: Numerical examples
% -----------------------------------------------------------------------------------------

\section{Numerical examples}\label{sec:numerical_examples}

In this section we  illustrate the performance of the $hp$-adaptive procedure outlined in Algorithm~1 with some numerical experiments in 1D and 2D.

% -----------------------------------------------------------------------------------------
%    Subsection: 1D examples
% -----------------------------------------------------------------------------------------

\subsection{Numerical examples in 1D}

In the following $1$-dimensional examples on the domain $\Omega := (0,1)$ we use a basis for the $hp$-finite element spaces $\WW$ that consists of the usual hat-functions, and, on each element $Q\in\Q$ with $p_Q\geq 2$, of the (elementwise transformed) integrated Legendre polynomials  given by $\psi_j^Q := \psi_j\circ F_Q^{-1}$, for $
 2 \leq j \leq p_Q$, cf.~\eqref{eq:defi_1D_psi}. In accordance with \eqref{eq:defi_xi_p_enrichment} we consider the (extended) functions
\begin{align*}
 \xi_j^Q(x) := 
 \begin{cases} 
  \psi_j^Q(x), & \text{if } x\in Q, \\
  0, & \text{if } x\in\Omega\setminus Q.
 \end{cases}
\end{align*}
For each element $Q\in\Q$ of the current mesh, the locally supported subspace $\WWloc_Q\subset\WW$ is chosen as
\begin{align*}
 \WWloc_Q := \big\lbrace \xi_j^Q : 2\leq j \leq p_Q \big\rbrace.
\end{align*}
In particular, the corresponding linear projection operator $\Ploc_Q: \WW\to\WWloc_Q$ only retains all higher-order modes on~$Q$. Specifically, we choose
\[
 \mathfrak{E}_p^Q := \big\lbrace \xi_j^Q  : j\in\myind{J}_1^Q \big\rbrace, \qquad
 \text{with} \quad
 \myind{J}_1^Q := \lbrace 2,\ldots,p_Q+1 \rbrace,
\]
as \emph{$p$-enrichment} functions, i.e.~we increase the polynomial degree $p_Q$ to $p_Q+1$, thereby using a number of $p_Q$ locally supported degrees of freedom on~$Q$. Moreover, for an $hp$-\emph{replacement} on $Q$, we divide the element $Q$ into two equally sized subelements $T_0,T_1$, for which we allocate some local polynomial degrees $p_0,p_1$, respectively, that give rise to $p_0 + p_1 + 1$ degrees of freedom within~$Q$: 
\begin{center}
	\includegraphics[scale=0.4]{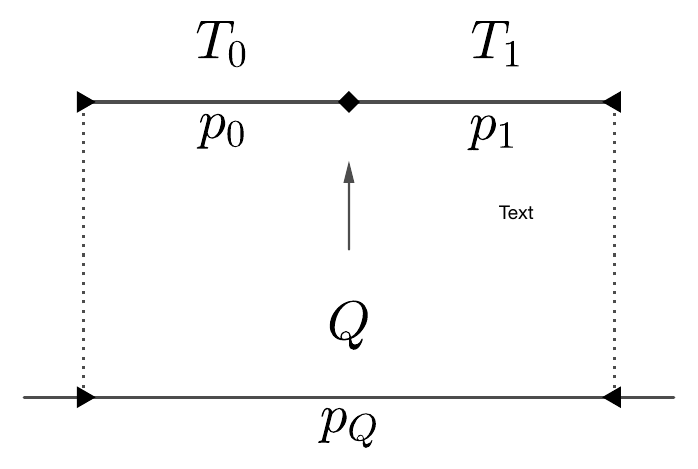}
\end{center}
In order to compare \emph{competitive $p$- and $hp$-refinements}, we impose the constraint
\begin{align}\label{eq:formula_competitive}
 p_0 + p_1 = p_Q + 1.
\end{align}
The only internal node for the refinement $\R(Q) = \lbrace T_0,T_1\rbrace$ of the element $Q$ is the $0$-dimensional midpoint, represented by $\myind{n}_1 := ((),())$. The subelements $T_1$, $T_2$ are characterized by the tuples $\myind{n}_2 := ((1),(0))$ and $\myind{n}_3 := ((1),(1))$. Therefore, in terms of the notation from~\S\ref{subsec:hp_enrichments}, we have $\N = \lbrace \myind{n}_1, \myind{n}_2, \myind{n}_3 \rbrace$, and we choose
\begin{align*}
 \myind{P}^{\R(Q)}(\myind{n}_1) = \big\lbrace () \big\rbrace, \qquad
 \myind{P}^{\R(Q)}(\myind{n}_2) = \big\lbrace 2,\ldots,p_1 \big\rbrace, \qquad
 \myind{P}^{\R(Q)}(\myind{n}_3) = \big\lbrace 2,\ldots,p_2 \big\rbrace,
\end{align*}
for all possible pairs $(p_1,p_2)$ of local polynomial degree combinations that satisfy \eqref{eq:formula_competitive}.

% -----------------------------------------------------------------------------------------
%    Subsubsection: First example

\subsubsection{Singularly perturbed problem with boundary layers}

As a first example, for (a possibly small) parameter $\varepsilon>0$, we consider the $1$-dimensional singularly perturbed differential equation
$
 -\varepsilon \, u'' + u = 1 $ in the domain $ \Omega=(0,1),
$
with the homogeneous Dirichlet boundary condition $u(0) = u(1) = 0$. The analytic solution is given by
\begin{align*}
 u(x) = \frac{e^{-c} - 1}{e^{c} - e^{-c}} \, e^{cx} + \frac{1 - e^{c}}{e^{c} - e^{-c}} \, e^{-cx} + 1, 
\end{align*}
with $c= \varepsilon^{-\nicefrac12}$; for very small $0<\varepsilon \ll 1$, we notice that $u$ exhibits thin boundary layers in the vicinity of the boundary points $x=0$ and $x=1$, and takes values of approximately~1 in the interior of the domain $\Omega$. For our numerical experiments, we initiate the $hp$-adaptive procedure of Algorithm~1 with a coarse mesh $\Q_0$ consisting of $4$ elements, and an associated uniform polynomial degree of $1$ on each of them. For the marking process we choose the D\"{o}rfler marking parameter to be~$\nicefrac12$. 

Figure~\ref{fig:resulting_meshHuston} shows the resulting $hp$-mesh after $28$ $hp$-adaptive steps for $\varepsilon = 10^{-5}$; we clearly see that the boundary layers have been resolved on a few small elements (of high polynomial degree), whilst no refinement is employed in the interior of the domain $\Omega$, where the solution is nearly constant.
\begin{figure}
	\centering
	\includegraphics[scale=0.15]{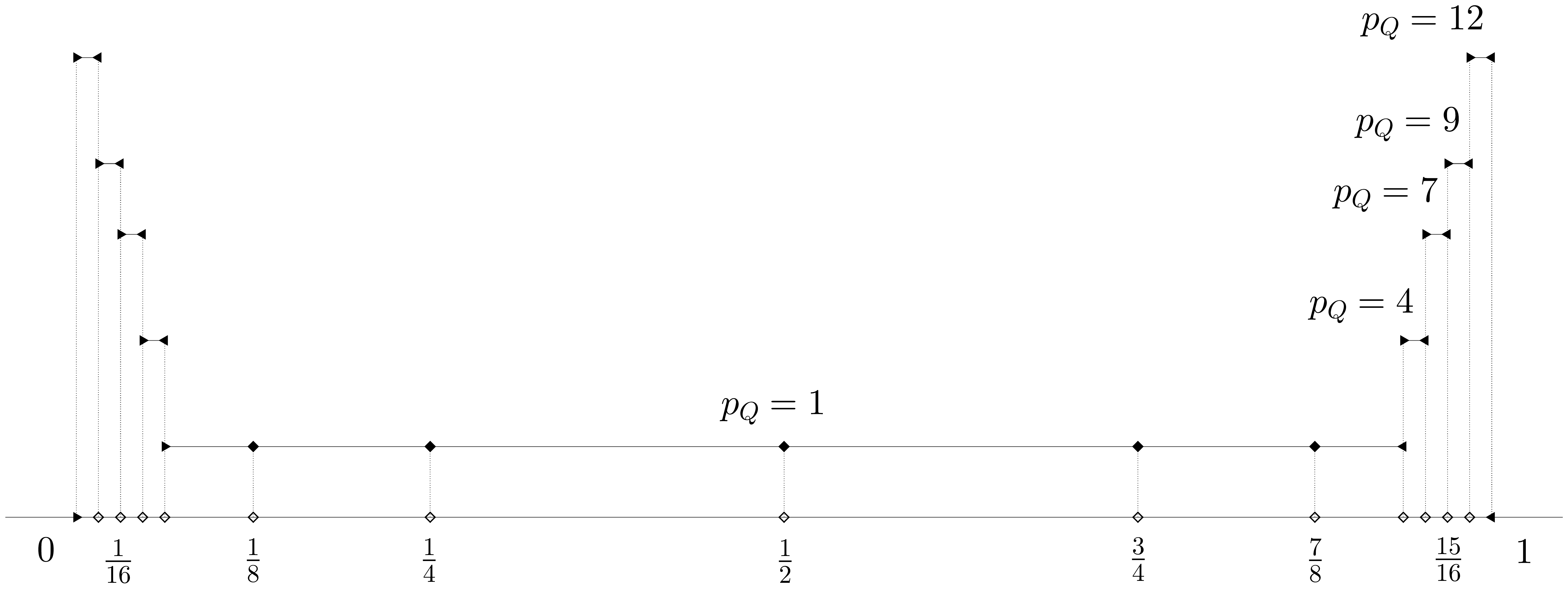}
	\caption{\it $hp$-mesh $\Q_{28}$ after $28$ adaptive enrichment steps with $\varepsilon = 10^{-5}$.}
	\label{fig:resulting_meshHuston}
\end{figure}
Moreover, in Figure~\ref{fig:convergence_Huston}, for the underlying energy norm given by
\begin{align*}
 \norm{v}^2 = \varepsilon \, \Vert v'\Vert_{L^2(\Omega)}^2 + \Vert v\Vert_{L^2(\Omega)}^2,
 \qquad v\in\XX,
\end{align*}
the error $\norm{ u - u_\WW}$ is plotted  with respect to the number of degrees of freedom in a semi-logarithmic scaling for several values of the singular perturbation parameter~$\varepsilon\in\{10^{-j}:\,j=3,4,5\}$. We observe that the $hp$-adaptive procedure is able to achieve an exponential rate of convergence that is fairly robust with respect to $\varepsilon$. In this regard, our results are comparable to alternative $hp$-adaptive strategies proposed in the literature, see, e.g.,~\cite[Expl.~2]{wihler:2011}.

\begin{figure}
	\centering
	\includegraphics[scale=0.35]{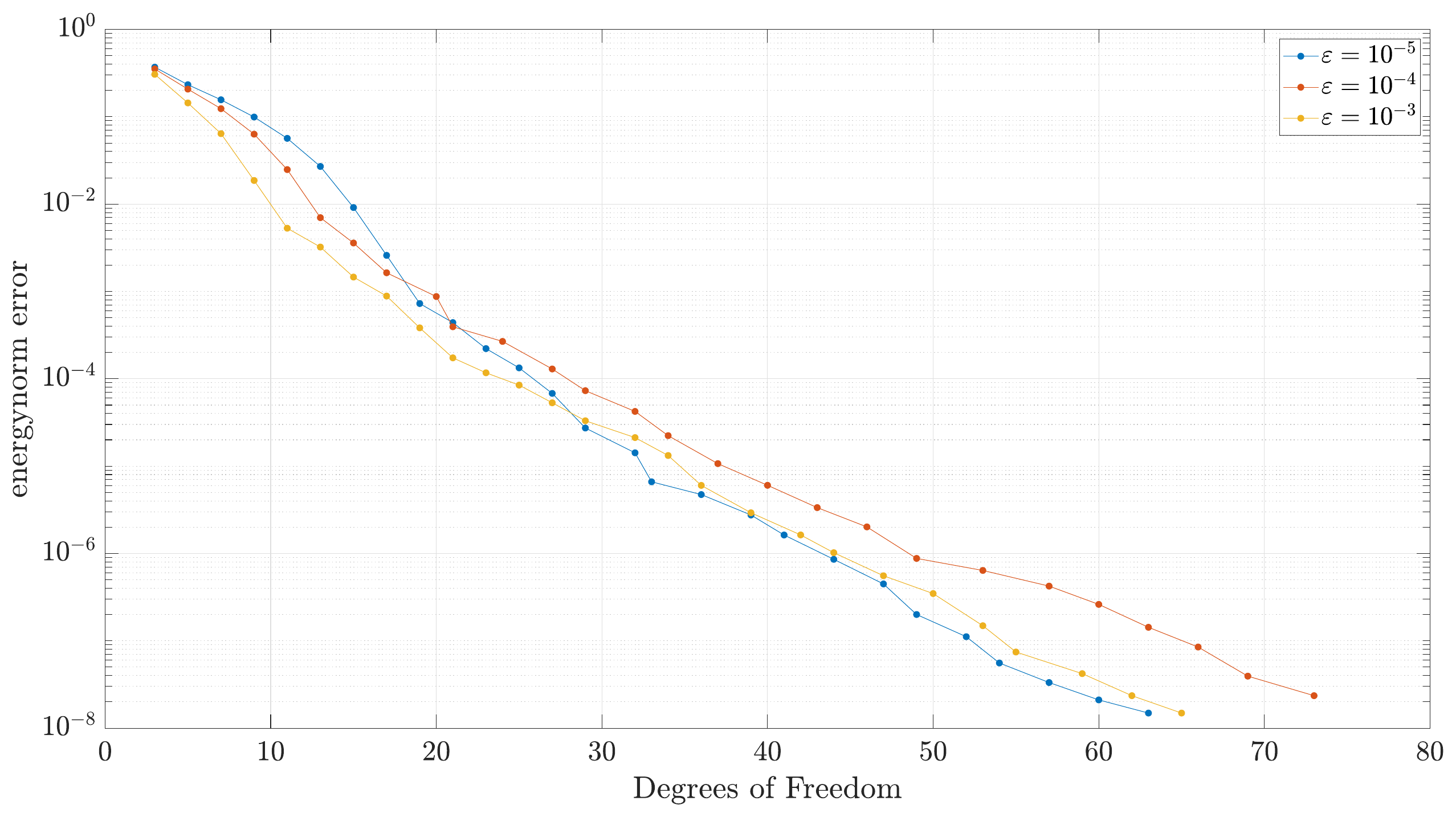}
	\caption{\it Singularly perturbed problem in 1D: Performance of the $hp$-adaptive procedure with respect to the energy norm error $\norm{u - u_{\WW}}$ for $\varepsilon = 10^{-j}$ with $j\in\lbrace 3,4,5\rbrace$.}
	\label{fig:convergence_Huston}
\end{figure}

% -----------------------------------------------------------------------------------------
%    Subsubsection: First example

\subsubsection{1D-model problem with a boundary singularity}

As a second example, we consider  the $1$-dimensional boundary value problem
$ - u'' = f$ in $\Omega=(0,1),$
with the homogeneous Dirichlet boundary conditions $u(0)=u(1)=0$; the right-hand side function $f$ is chosen in such a way that the analytic solution is given by $u(x) = x^{\nicefrac34} - x$, which features a singularity at~$x=0$ with $u'(x)\to\infty$ as $x\to0^+$. As in the previous example, we start the $hp$-adaptive procedure of Algorithm~1 with an initial mesh $\Q_0$ consisting of $4$ elements with a uniform polynomial degree distribution of $1$, and let the D\"{o}rfler marking parameter be $\nicefrac12$. The ensuing plot shows the resulting $hp$-mesh $\Q_{49}$, which containts 51 elements, after $49$ $hp$-adaptive enrichment steps:
%$Q_0,\ldots, Q_{50}$ (enumerated from the left to the right) 
%\begin{align*}
% p_1 = 3, \quad
% p_2 = 1, \quad
% p_3,\ldots,p_{7} = 2, \quad
% p_{8},\ldots, p_{19} = 3, \quad
% p_{20},\ldots, p_{31} = 4, \quad
% p_{32},\ldots, p_{42} = 5, \quad
% p_{43},\ldots, p_{49} = 6, \quad
% p_{50} = 5, \quad
% p_{51} = 4.
%\end{align*}
%\begin{align*}
% p_j = \begin{cases}
%  1, & \text{for } j = 1, \\
%  2, & \text{for } 2\leq j\leq 6, \\
%  3, & \text{for } 7\leq j\leq 18 \text{ and } j=0, \\
%  4, & \text{for } 19\leq j\leq 30 \text{ and } j=50, \\
%  5, & \text{for } 31\leq j\leq 41 \text{ and } j=49, \\
%  6, & \text{for } 42\leq j\leq 48.
% \end{cases}
%\end{align*}
\begin{center}
\includegraphics[scale=0.3]{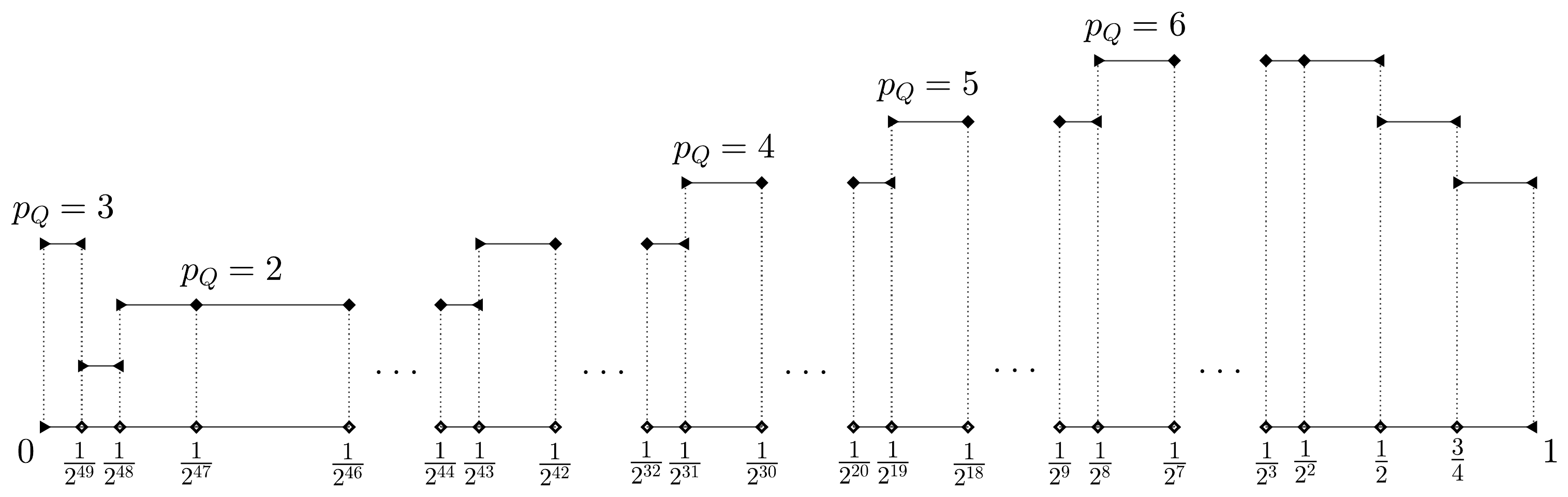}
\end{center}
The mesh is geometrically refined towards the singularity at $x=0$, and the polynomial degree is  increased at an approximately linear rate in dependence of the distance from the origin; this is in line with a priori results on exponentially convergent $hp$-FEM for local algebraic singularities; see, e.g., \cite{schwab:1998}. Indeed, this is confirmed in Figure~\ref{fig:convergence_1DPoisson}, where the error in the energy norm, i.e. $\norm{ u - u_\WW}=\Vert u'-u_{\WW}'\Vert_{L^2(\Omega)}$, shows an exponential decay (with respect to the square root of the number of degrees of freedoms) in a semilogarithmic scaling. 

\begin{figure}
	\centering
	\includegraphics[scale=0.33]{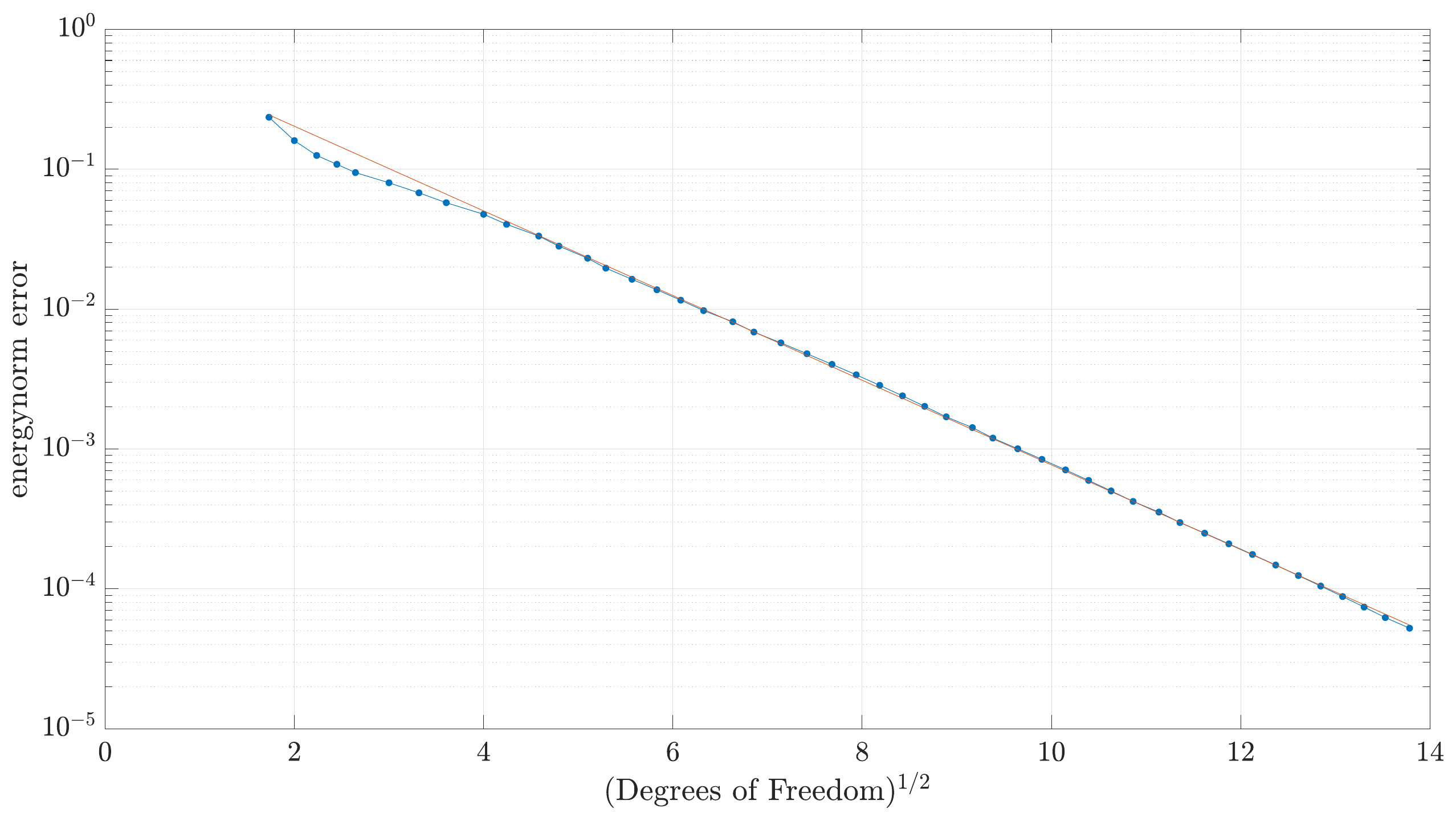}
	\caption{\it Algebraic boundary singularity in 1D: Performance of the $hp$-adaptive procedure with respect to the energy norm error $\norm{u - u_{\WW}}$ (measured against the square root of the number of degrees of freedom).}
	\label{fig:convergence_1DPoisson}
\end{figure}

% -----------------------------------------------------------------------------------------
%    Subsection: 2D examples
% -----------------------------------------------------------------------------------------

\subsection{Numerical example in 2D}

In the following $2$-dimensional example over the unit square $\Omega := (0,1)^2$, we use a basis for the $hp$-finite element spaces $\WW$ that---in addition to the usual hat functions---contains the functions
\begin{align*}
 \psi_{\myind{j}}^Q := \myref{\psi}_{\myind{j}} \circ F_Q^{-1}, \qquad
 \myind{j} \in \lbrace 2,\ldots, p_Q\rbrace^2,
\end{align*}
extended by $0$ to $\Omega$, cf.~\eqref{eq:defi_dDim_psi}, for elements $Q\in\Q$ with a local polynomial degree $p_Q\geq 2$. In accordance with \eqref{eq:defi_xi_p_enrichment} we  denote the corresponding extended functions by $\xi_{\myind{j}}^Q$. Recall that we decompose the $hp$-finite element solution $u_{\WW}\in\WW$ of \eqref{eq:solW} into a local part $\uloc$ and a remaining part $\ut$, cf.~\eqref{eq:dw}. If we wish the remaining part $\ut$ to be completely independent of the enrichments on $Q$ (in the sense that $\ut$ remains unchanged under modifications in $\WWloc_Q$), we note that the locally supported subspace $\WWloc_Q$ needs to contain all basis functions of $\WW$ with support on $Q$; in particular, all interior bubble functions on $Q$. For simplicity (see Remark~\ref{rem:patches}), we set
\begin{align*}%\label{eq:defi_Wloc_2D}
 \WWloc_Q := \big\lbrace \xi_{\myind{j}}^Q : \myind{j} \in \lbrace 2,\ldots, p_Q\rbrace^2 \big\rbrace,
\end{align*}
and, on each element $Q\in\Q$ (with an associated local polynomial degree $p_Q$), we choose the $p$-enrichment functions to be
\begin{align}\label{eq:2D_enrichment}
 \mathfrak{E}_p^Q := \big\lbrace \xi_{\myind{j}}^Q  : \myind{j}\in\myind{J}_2^Q \big\rbrace, \qquad
 \text{with} \quad
 \myind{J}_2^Q := \lbrace 2,\ldots,p_Q+1 \rbrace^2,
\end{align}
i.e. we increase the local polynomial degree $p_Q$ to $p_Q+1$ in both coordinate directions, thereby resulting in $p_Q^2$ many locally supported enrichment functions on~$Q$. Moreover, an $hp$-refinement on an element $Q$ is based on dividing $Q$ into the four subelements $T_{\myind{i}}$ with $\myind{i}\in\lbrace 0,1\rbrace^2$:
\begin{center}
	\includegraphics[scale=0.5]{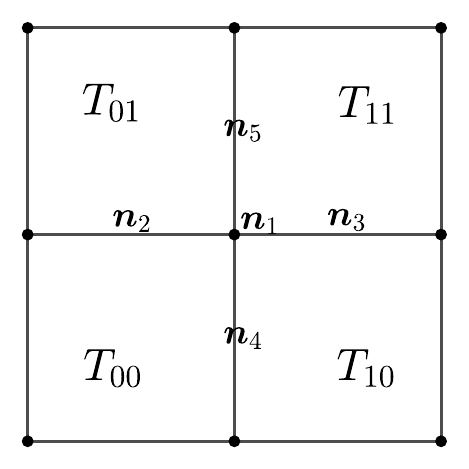}
\end{center}
In the present experiments we limit ourselves to enrichments that are restricted to single element $Q$ as pointed out in \S\ref{subsec:enrichmentF}, and do not involve enrichments on a patch around $Q$; see Remark~\ref{rem:patches} below for more details on this matter. For simplicity, we compare the $p$-enrichment \eqref{eq:2D_enrichment} with an $hp$-refinement on $Q$ that features the same polynomial degree $p_{Q}$ on all subelements $T_{\myind{i}}\in\R(Q)$, as it is done in most $hp$-adaptive algorithms.
The internal nodes of the refinement $\R(Q) = \lbrace T_{\myind{i}} : \myind{i}\in\lbrace 0,1\rbrace^2 \rbrace$ consist of the midpoint $\myind{n}_1 := ((),())$, of the edges
\begin{align*}
 \myind{n}_2 := ((1),(0)), \quad
 \myind{n}_3 := ((1),(1)), \quad
 \myind{n}_4 := ((2),(0)), \quad
 \myind{n}_5 := ((2),(1)),
\end{align*}
and of the subelements $T_{\myind{i}}$, represented by
\begin{align*}
 \myind{n}_6 := ((1,2),(0,0)), \quad 
 \myind{n}_7 := ((1,2),(1,0)), \quad 
 \myind{n}_8 := ((1,2),(0,1)), \quad 
 \myind{n}_9 := ((1,2),(1,1)).
\end{align*}
For the adaptive procedure, we set
\begin{align*}
 \myind{P}^{\R(Q)}(\myind{n}_j) = \begin{cases}
  \big\lbrace () \big\rbrace, & \text{for } j=1, \\
  \big\lbrace (p_1) : p_1\in\lbrace 2,\ldots,p_Q\rbrace \big\rbrace, & \text{for } 2\leq j\leq 5, \\
  \big\lbrace (p_1, p_2) : p_1,p_2 \in\lbrace 2,\ldots,p_Q\rbrace \big\rbrace, & \text{for } 6\leq j\leq 9,
 \end{cases}
\end{align*}
thereby leading to $1 + 4 \, (p_Q - 1) \, p_Q$ $hp$-enrichment functions on $Q$.

\begin{remark}[Competitive refinements for dimensions $d\ge 2$]\label{rem:patches}
In contrast to the $1$-dimensional case, if the remaining part $\ut$ of the solution $u_\WW$ should be independent (in the sense that $\ut$ remains unchanged upon modifications in $\WWloc_Q$) of an enrichment or a refinement associated with an element $Q$, it is mandatory to include element interface basis functions in the definition of the locally supported subspace $\WWloc_Q$. Indeed, only if degrees of freedom, which are shared by neighbouring elements, are taken into account as well, it is possible to compare different $p$-enrichments and $hp$-refinements in a truly competitive way. This is due to the fact that, in general, $p$-enrichments and $hp$-refinements on a single element~$Q$ may influence the new solution on a local patch which also involves neighbouring elements of~$Q$; we observe that this effect extends even beyond the direct neighbours of~$Q$ if hanging nodes, as can be seen in Figure~\ref{fig:patch_Q}, are present. In turn, we note that the enforcement of continuity of enrichment functions along element edges and faces of a patch around $Q$ (e.g.~by means of the minimum rule for the adjacent polynomial degrees) causes an additional challenge. For simplicity, for the purpose of the present paper, we will \emph{restrict ourselves to the enrichment of internal degrees of freedom only on each element}; competitive enrichments on patches for higher dimension are investigated in forthcoming work.
\end{remark}

\begin{figure}
	\centering
	\includegraphics[scale=0.4]{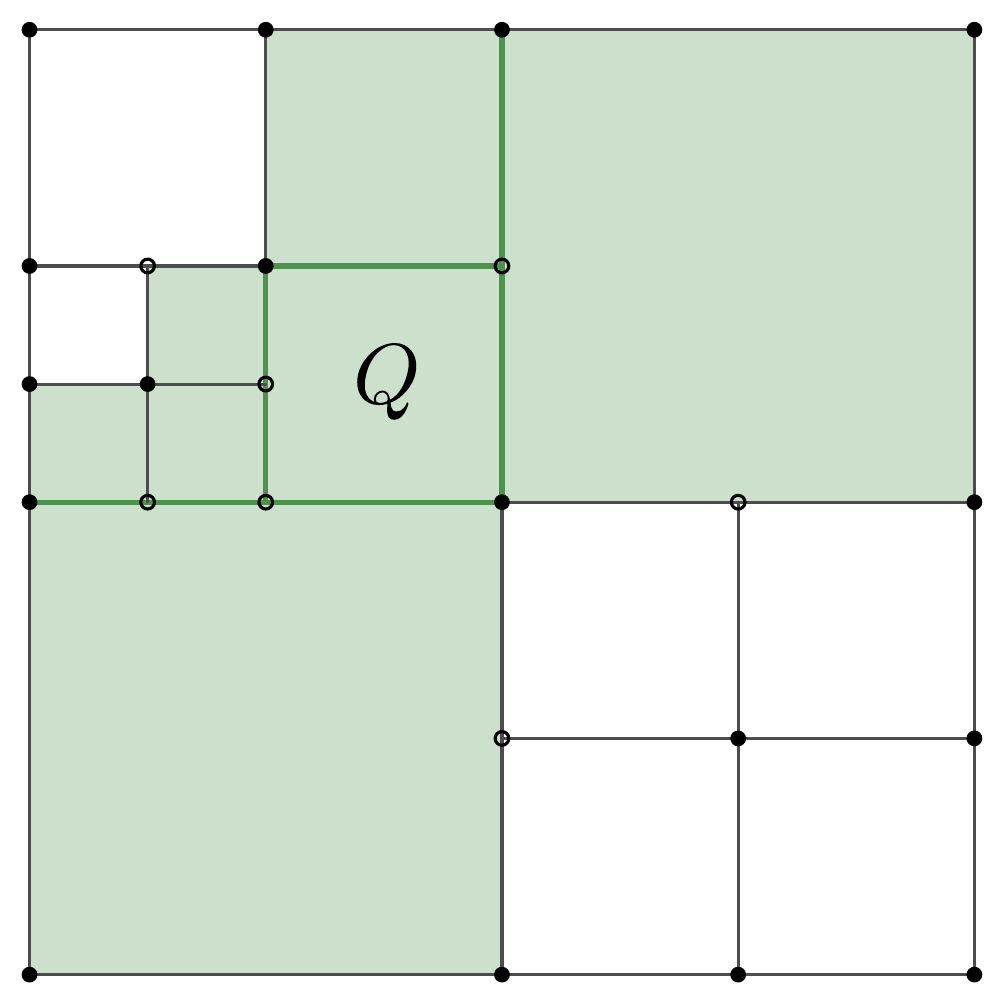}
	\caption{\it Extended patch around the element $Q$.}
	\label{fig:patch_Q}
\end{figure}

% -----------------------------------------------------------------------------------------
%    Subsubsection: First example

\subsubsection{2D-Poisson problem with corner singularities}

Let us consider the $2$-dimensional Poisson-problem
$ -\Delta u = 1 $ on $ \Omega=(0,1)^2$,
subject to the homogeneous Dirichlet boundary condition $u=0$ on $\Gamma := \partial\Omega$. While an explicit expression for the analytic solution $u$ is unavailable, an eigenfunction expansion yields that
\begin{align*}
 \norm{u}^2
 = \Vert \nabla u \Vert_{L^2(\Omega)}^2
 = \left( \frac{2}{\pi} \right)^6 \sum_{k,l\geq 1 \, \text{odd}} \frac{1}{k^2 \, l^2 \, (k^2 + l^2)}
 \approx 0.035144253738788451\ldots
\end{align*}
We start the $hp$-adaptive procedure of Algorithm~1 with an initial mesh $\Q_0$ consisting of $16$ elements with a uniform polynomial degree distribution of $1$ on all elements $Q\in\Q_0$. Moreover, we let the D\"{o}rfler parameter to be $\nicefrac14$. Figure~\ref{fig:2D_mesh_poisson} shows the resulting $hp$-mesh after $29$ adaptive enrichment steps (containing $124$ elements).
\begin{figure}
	\centering
	\includegraphics[scale=0.7]{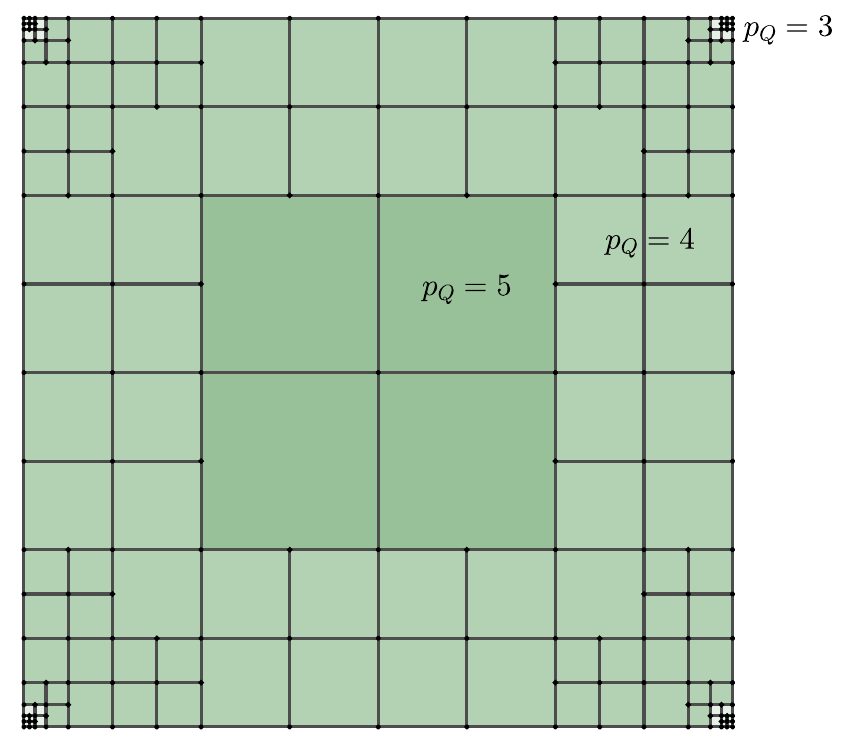}
	\caption{\it $hp$-mesh $\Q_{29}$ after $29$ adaptive enrichment steps.}
	\label{fig:2D_mesh_poisson}
\end{figure}
Following the a priori error analysis on the exponential convergence of $hp$-FEM for the 2-dimensional Poisson equation with corner singularities in polygons, see, e.g.~\cite{schwab:1998}, we depict the energy error $\norm{u - u_{\XX}}$  with respect to the 3rd root of the number of degrees of freedom in Figure~\ref{fig:convergence_2DPoisson}. The results show that our $hp$-adaptive algorithm is able to properly resolve the four corner singularities on geometrically refined meshes, and that approximately exponential rates of convergence can be achieved.

\begin{figure}
	\centering
	\includegraphics[scale=0.33]{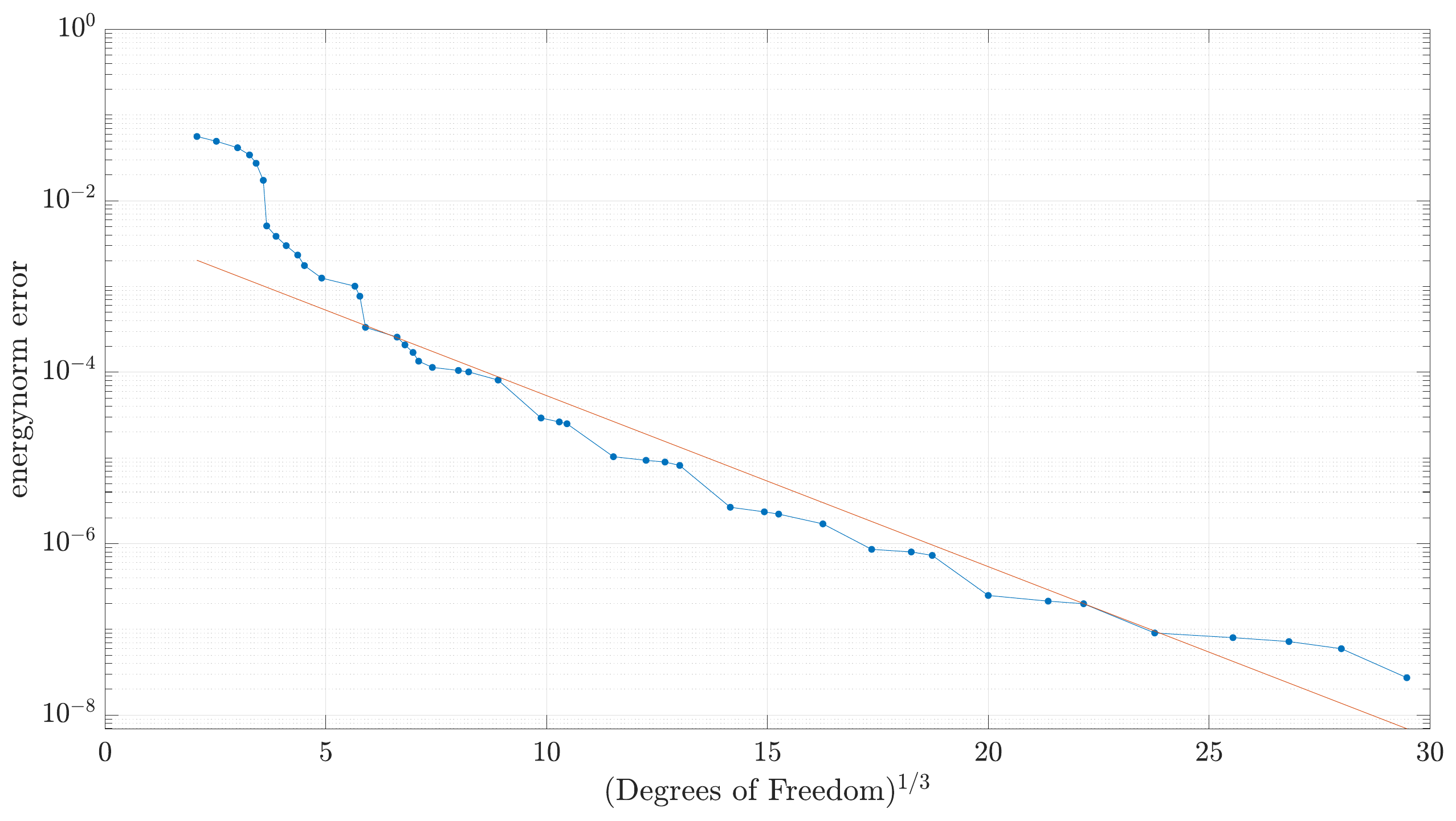}
	\caption{\it Poisson problem with corner singularities on unit square: Performance of the $hp$-adaptive procedure with respect to the energy norm error $\norm{u - u_{\WW}}$ (measured against the 3rd root of the number of degrees of freedom).}
	\label{fig:convergence_2DPoisson}
\end{figure}

% -----------------------------------------------------------------------------------------
%    Appendix
% -----------------------------------------------------------------------------------------

\section*{Appendix}

We present some recursion formulas for the computation of the $1$-dimensional constraint coefficients applied in~\S\ref{subsec:constraint_coeff}. For this purpose, let $I=[a,b]$, with $-1\leq a < b \leq 1$, be an interval. For $j\in\NN_0$, recall that the restrictions of the functions $\psi_j:[-1,1]\to\RR$ from~\eqref{eq:defi_1D_psi} to the interval $I$ can be represented as linear combinations of the functions $\wt{\psi}_i:=\psi_i \circ F_I^{-1}$ on $I$, for $i\in \underline{j}_0$:
\begin{align*}
 \psi_{i \, | \, I} = \sum_{j=0}^{i} b_{i,j}^{I} \, \wt{\psi}_i, \qquad
 i\in\NN_0;
\end{align*}
here,  the bijective affine linear transformation $F_I:[-1,1]\to I$ is given by $F_I(t) = \alpha \, t + \beta$, with
$ \alpha := \nicefrac{1}{2}(b-a)$ and 
$ \beta  := \nicefrac{1}{2}(a+b)$.
We note that the uniquely determined \emph{constraint coefficients} $b_{i,j}^{I}$ in the above representation can be determined recursively (see~\cite{byfut:2017} for a proof):
\begin{enumerate}[(i)]
\item For $i,j\in\lbrace 0,1\rbrace^2$, the following identities can be derived:
\begin{align*}
 b_{0,0}^{I} = \frac{1}{2} \, (1+\alpha-\beta), \quad
 b_{1,0}^{I} = \frac{1}{2} \, (1-\alpha+\beta), \quad
 b_{0,1}^{I} = \frac{1}{2} \, (1-\alpha-\beta), \quad
 b_{1,1}^{I} = \frac{1}{2} \, (1+\alpha+\beta).
\end{align*}

\item Moreover, we find that
\begin{align*}
 b_{2,0}^{I} = \frac{1}{2} \, \Big( (\alpha-\beta)^2 - 1 \Big), \qquad
 b_{2,1}^{I} = \frac{1}{2} \, \Big( (\alpha+\beta)^2 - 1 \Big), \qquad
 b_{2,2}^{I} = \alpha^2.
\end{align*}

\item For $i\geq 3$, it holds $ b_{i,i}^{I}   = \alpha \,  b_{i-1,i-1}^{I}$
as well as 
\begin{align*}
 b_{i,0}^{I} &= \frac{1}{i} \, \Big( (2i - 3) \, (\beta - \alpha) \, b_{i-1,0}^{I} - (i-3) \, b_{i-2,0}^{I} \Big), \\
 b_{i,1}^{I} &= \frac{1}{i} \, \Big( (2i - 3) \, (\alpha + \beta) \, b_{i-1,1}^{I} - (i-3) \, b_{i-2,1}^{I} \Big), \\
 b_{i,2}^{I} &= \frac{1}{i} \, \bigg( (2i - 3) \, \Big( \alpha \, \Big( \tfrac{1}{5} \, b_{i-1,3}^{I} - \big( b_{i-1,0}^{I} - b_{i-1,1}^{I} \big) \Big) + \beta \, b_{i-1,2}^{I} \Big) - (i-3) \, b_{i-2,2}^{I} \bigg).
\end{align*}
In addition, for $i\geq 4$, we have
\begin{align*}
  b_{i,i-1}^{I} = \frac{2i-3}{i} \, \left( \frac{i-1}{2i-5} \, \alpha \, b_{i-1,i-2}^{I} + \beta \, b_{i-1,i-1}^{I} \right).
\end{align*}

\item For $i\geq 5$ and $j\in\lbrace 3,\ldots, i-2\rbrace$, the coefficients are given by
\begin{align*}
 b_{i,j}^{I} = \frac{1}{i} \, \bigg( (2i - 3) \, \Big( \alpha \, \Big( \tfrac{j}{2j-3} \, b_{i-1,j-1}^{I} + \tfrac{j-1}{2j+1} \, b_{i-1,j+1}^{I} \Big) + \beta \, b_{i-1,j}^{I} \Big) - (i-3) \, b_{i-2,j}^{I} \bigg).
\end{align*}

\item Finally, for $j\geq 2$ and $i\in\lbrace 0,\ldots, j-1\rbrace$, it holds $b_{i,j}=0$.
\end{enumerate}

% -----------------------------------------------------------------------------------------
%    Acknowledgment, Funding & Bibliography
% -----------------------------------------------------------------------------------------

\bibliographystyle{amsalpha}
%\bibliography{...}

\end{document}